\documentclass[12pt]{amsart}
\usepackage[T1]{fontenc}
\usepackage[utf8]{inputenc}
\usepackage{lmodern}
\usepackage{amsmath,amssymb,amsthm,mathrsfs}
\usepackage{microtype}
\usepackage[left=1.06in,right=1.06in,top=1.015in,bottom=1.03in]{geometry}
\usepackage{hyperref}
\hypersetup{colorlinks=true,linkcolor=black,citecolor=black,urlcolor=black}

\newtheorem{theorem}{Theorem}[section]
\newtheorem{lemma}[theorem]{Lemma}
\newtheorem{proposition}[theorem]{Proposition}
\newtheorem{corollary}[theorem]{Corollary}
\theoremstyle{remark}
\newtheorem{remark}[theorem]{Remark}

\newcommand{\D}{\mathbb D}
\newcommand{\C}{\mathbb C}
\newcommand{\R}{\mathbb R}
\newcommand{\psh}{\mathrm{PSH}}
\newcommand{\Ree}{\operatorname{Re}}
\newcommand{\LL}{\mathcal L}

\title[Kobayashi completeness and a smooth counterexample]{Kobayashi completeness of pseudoconvex domains: finite type and a smooth counterexample}
\author{Armen Edigarian}
\address{Faculty of Mathematics and Computer Science, Jagiellonian University, Prof. St. \L ojasiewicza 6, 30-348 Krak\'ow, Poland; ORCID: \href{https://orcid.org/0000-0003-1789-1287}{0000-0003-1789-1287}}
\email{armen.edigarian@uj.edu.pl}

\begin{document}
\begin{abstract}
We prove that every bounded pseudoconvex domain in $\C^N$ with
$C^\infty$ boundary of finite D'Angelo one-type is Kobayashi complete. As a separate consequence of our criterion of completeness and an estimate of Diederich--Forn\ae ss, we prove that every bounded
pseudoconvex domain with real-analytic boundary is Kobayashi complete. We then show that smoothness alone does not imply completeness: there exists a bounded pseudoconvex domain in $\C^3$ with $C^\infty$ boundary
and a global smooth plurisubharmonic defining function whose Kobayashi and Carath\'eodory distances are incomplete. 
\end{abstract}
\maketitle
\section{Introduction}
The Kobayashi distance of a bounded domain need not be complete, even
when the domain is pseudoconvex.  A natural question is which boundary
conditions force completeness.  For strongly pseudoconvex domains,
completeness follows from the estimates of Graham \cite{Graham}.  In
two complex dimensions, Bedford and Forn\ae ss \cite{BedfordForn}
constructed holomorphic peak functions at smooth pseudoconvex boundary
points of finite type; in particular, bounded smooth pseudoconvex
domains of finite type in $\C^2$ are complete for both the
Carath\'eodory and Kobayashi distances.

Jarnicki and Pflug \cite[Section~2.8, p.~123]{JarnickiPflugRevisited}
ask the corresponding completeness questions under a finite-type
assumption.  We answer the Kobayashi question in arbitrary dimension.
We use finite type in the sense of D'Angelo \cite{DAngelo}.  If
$p\in\partial\Omega$ and $r$ is a smooth local defining function, we
write
\[
 \Delta_1(p)=\sup_\gamma
 \frac{\operatorname{ord}_0(r\circ\gamma)}
      {\operatorname{ord}_0(\gamma-p)},
\]
where $\gamma$ ranges over nonconstant holomorphic curve germs with
$\gamma(0)=p$.  An order of vanishing may be infinite.

\begin{theorem}\label{thm:finite-type}
Every bounded pseudoconvex domain in $\C^N$ with $C^\infty$
boundary of finite D'Angelo type is Kobayashi complete.
\end{theorem}

The proof separates two issues which are often intertwined.  Section~\ref{sec:criteria}
contains the completeness criterion and the normal self-improvement argument.  If $r<0$ is a defining
function and one has a power estimate
\[
 \kappa_\Omega(z;X)\gtrsim
 \frac{\|X\|}{\delta_\Omega(z)^\varepsilon},
\]
then the analytic-disc argument of Section~\ref{sec:criteria} improves it
in the complex-normal direction to
\begin{equation}\label{eq:intro-normal}
 \kappa_\Omega(z;X)\ge
 c\,
 \frac{\left|\sum_{j=1}^N r_{z_j}(z)X_j\right|}
 {(-r(z))\log(A/(-r(z)))}
\end{equation}
near the boundary.  Consequently, for each fixed $z_0\in\Omega$,
\begin{equation}\label{eq:intro-distance}
 k_\Omega(z_0,z)\ge
 c\log\log\frac{A}{\delta_\Omega(z)}-C_{z_0}.
\end{equation}
The proof is first given when $r$ is plurisubharmonic.  In that case
$-r\circ f$ is positive superharmonic for every analytic disc
$f:\D\to\Omega$, so the required mean estimate is immediate.  For a
general smoothly bounded pseudoconvex domain, the same proof goes
through after replacing this superharmonicity by a uniform disc mean
estimate derived directly from pseudoconvexity.

The same criterion also gives a short proof of the real-analytic case.
Diederich and Forn\ae ss \cite[Theorem~4, p.~577]{DiederichFornaess79}
proved a power lower bound for the Kobayashi--Royden metric on bounded
pseudoconvex domains with real-analytic boundary.  Combining their
estimate with the self-improvement in Section~\ref{sec:criteria} gives:

\begin{corollary}\label{cor:real-analytic-intro}
Every bounded pseudoconvex domain in $\C^N$ with real-analytic boundary
is Kobayashi complete.
\end{corollary}

It remains, for Theorem~\ref{thm:finite-type}, to obtain the initial
power estimate at finite-type points.  This is done in
Section~\ref{sec:finite-type-power}.  An estimate of this form was sought
in Cho's paper \cite{Cho}.  We do not use Cho's argument.  Instead, the
proof given here couples the uniform disc mean estimate with a
finite-jet polarization, a unitary family of holomorphic ideals, and
D'Angelo's type--multiplicity inequality \cite[Theorem~2.7]{DAngelo}.  Once the power estimate is established locally, a
finite covering of the boundary and the completeness criterion prove
Theorem~\ref{thm:finite-type}.

The finite-type hypothesis is essential for this conclusion.  The
smooth-boundary questions in the \emph{Several Complex Variables
Problem List} \cite[Problems~7.3.1 and 7.5.1]{ProblemList} ask whether
bounded smoothly bounded pseudoconvex domains must be complete.  For
the Kobayashi distance the answer is negative, even in the presence of
a global smooth plurisubharmonic defining function.

\begin{theorem}\label{thm:main}
There exists $\rho\in C^\infty(\C^2)\cap\psh(\C^2)$ such that
\[
 \Omega=\{(s,t,w)\in\C^3:|w|^2+\rho(s,t)<1\}
\]
is a bounded domain with $C^\infty$ boundary and incomplete Kobayashi
distance.  Moreover, $\Omega$ admits a Kobayashi-Cauchy sequence
converging Euclideanly to $(0,0,1)\in\partial\Omega$ along the inward
normal, and its Carath\'eodory distance is also incomplete.
\end{theorem}

Thus $C^\infty$ boundary regularity, in contrast with finite type, does
not force Kobayashi completeness.  The example has an additional
feature.  Its global defining function
\[
 r(s,t,w)=|w|^2+\rho(s,t)-1
\]
is plurisubharmonic, negative on $\Omega$, and tends to $0$ at the
boundary.  Hence $r|_\Omega$ is a bounded negative plurisubharmonic
exhaustion, so $\Omega$ is hyperconvex.  By B\l ocki--Pflug
\cite{BlockiPflug}, and independently Herbort \cite{HerbortBergman},
every bounded hyperconvex domain is Bergman complete.  Thus the same
smoothly bounded domain is Bergman complete while both its Kobayashi and
Carath\'eodory distances are incomplete.  The boundary point at which
incompleteness occurs has infinite D'Angelo type, in accordance with
Theorem~\ref{thm:finite-type}.

The construction modifies the analytic discs used by Forn\ae ss and
Lee \cite{FornLee}, following earlier work of Krantz \cite{Krantz}; see
also Pflug and Zwonek \cite{PflugZwonek}.  Their construction uses a
cuspidal curve $\xi\mapsto(\xi^3,\xi^2)$.  Here we use instead the nodal polynomial curve
\[
 \Gamma(\xi)=(\xi^2(\xi-1),\xi(\xi-1)).
\]
The identities $\Gamma(0)=\Gamma(1)=0$ permit an analytic disc to meet
the same complex normal line at two different parameter values.
Rescaling such discs and choosing their normal components appropriately
produces a sequence of points approaching the boundary whose successive
Kobayashi distances are summable.  The main work lies in constructing
a smooth plurisubharmonic function with prescribed restrictions to
these discs.  The distinction from the classical nonsmooth examples is
global boundary regularity.  Sibony's unpublished example, as described
by Jarnicki and Pflug \cite[Section~2.8]{JarnickiPflugRevisited}, has
$C^\infty$ boundary except at one point, and the recent examples of
Masanta \cite{Masanta} have the same regularity pattern.  Earlier
balanced counterexamples were obtained in
\cite{JarnickiPflug,Edigarian2}.

The paper is organized as follows.  Section~\ref{sec:criteria} states
the log--log completeness criterion and proves the normal self-improvement,
first in the presence of a plurisubharmonic defining function and then for
general pseudoconvex domains; it ends with the proof of the real-analytic
corollary from the Diederich--Forn\ae ss estimate.  Section~\ref{sec:finite-type-power}
proves the finite-type power lower bound and concludes the proof of
Theorem~\ref{thm:finite-type}.  Sections~4--6 construct the domain in
Theorem~\ref{thm:main}: Section~4 builds the scalar subharmonic functions
and polynomial discs, Section~5 constructs the smooth plurisubharmonic
potential, and Section~6 assembles the domain and proves Kobayashi and
Carath\'eodory incompleteness.

We write $\D=\{\zeta\in\C:|\zeta|<1\}$.  Our normalization is
\[
 \kappa_\Omega(z;X)=\inf\{a>0:\ f\in\mathcal O(\D,\Omega),\
 f(0)=z,\ f'(0)=a^{-1}X\}.
\]
In Sections~4--6 we use $(s,t)$ for coordinates on $\C^2$ and $w$ for
the third coordinate on $\C^3$.  We write $k_\Omega$ for the associated
Kobayashi distance and $c_\Omega$ for the Carath\'eodory distance.  Thus
holomorphic discs decrease the Poincar\'e distance
$\operatorname{arctanh}|(\zeta-\xi)/(1-\bar\xi\zeta)|$, and
$c_\Omega\le k_\Omega$.  Completeness always means Cauchy completeness.
All $C^m$ norms use real derivatives; a smooth function is flat at the
origin when all its derivatives vanish there.  Constants denoted by $C$
may change from line to line, but depend only on fixed quantities unless
otherwise indicated.

\section{Kobayashi completeness criterion and self-improvement}\label{sec:criteria}
\subsection{The completeness criterion}

The following log--log criterion is the form of the completeness principle that we need; compare Fu~\cite[Remark following Theorem~2.6]{Fu}.

\begin{proposition}[Log--log completeness criterion]\label{prop:criterion}
Let $\Omega\subset\C^N$ be a bounded domain and let
\[
r\in C^1(\Omega)\cap C(\overline\Omega),
\qquad
r<0\ \text{on }\Omega,
\qquad
r=0\ \text{on }\partial\Omega.
\]
Suppose that for some $c>0$, $A>0$, and $t_0>0$,
\begin{equation}
\kappa_\Omega(z;X)
\ge
c\,
\frac{\left|\displaystyle\sum_{j=1}^N r_{z_j}(z)X_j\right|}
{(-r(z))\log\!\dfrac{A}{-r(z)}}
\label{eq:crit-loglog}
\end{equation}
whenever $0<-r(z)<t_0$ and $X\in\C^N$. Then $\Omega$ is Kobayashi complete.

More precisely, after decreasing $t_0$ if necessary, there is $C>0$ such that
\begin{equation}
k_\Omega(z,w)
\ge
C\left|
\log\log\frac{A}{-r(z)}
-
\log\log\frac{A}{-r(w)}
\right|
\label{eq:crit-distance-pair}
\end{equation}
whenever $z$ and $w$ are sufficiently close to $\partial\Omega$.
For every fixed $z_0\in\Omega$,
\begin{equation}
k_\Omega(z_0,z)
\ge
C\log\log\frac{A}{-r(z)}-C_{z_0}
\label{eq:crit-distance-point}
\end{equation}
for $z$ sufficiently close to $\partial\Omega$.

If in addition
\[
A>\max_{\overline\Omega}(-r),
\]
then there is $C>0$ such that
\begin{equation}
k_\Omega(z,w)
\ge
C\left|
\log\log\frac{A}{-r(z)}
-
\log\log\frac{A}{-r(w)}
\right|
\label{eq:crit-distance-global}
\end{equation}
holds for all $z,w\in\Omega$.
\end{proposition}

\begin{proof}
Decrease $t_0$ so that $t_0<A$, and choose $t_1\in(0,t_0)$. Put
\[
h(t)=\log\log\frac At,
\qquad 0<t\le t_1.
\]
Let
\[
M=\max_{\overline\Omega}(-r)>0,
\]
and extend $h$ from $(0,t_1]$ to a $C^1$ function
$\widetilde h:(0,M]\to\mathbb R$ whose derivative is bounded on
$[t_1,M]$. Set
\[
H(z)=\widetilde h(-r(z)).
\]
If $0<-r(z)<t_1$, then for a real tangent vector $V$, identified with a vector of $\C^N$,
\[
|dH_z(V)|
=
\frac{|dr_z(V)|}{(-r(z))\log(A/(-r(z)))}.
\]
Since
\[
dr_z(V)=2\operatorname{Re}\left(\sum_{j=1}^N r_{z_j}(z)V_j\right),
\]
\eqref{eq:crit-loglog} gives
\begin{equation}
|dH_z(V)|\le C\kappa_\Omega(z;V)
\label{eq:crit-dH}
\end{equation}
near the boundary. On the compact set
\[
\{z\in\Omega:-r(z)\ge t_1\},
\]
$dH$ is bounded, while the Kobayashi--Royden metric dominates a positive multiple of the Euclidean norm because $\Omega$ is bounded. Thus, after increasing $C$, \eqref{eq:crit-dH} holds for all $z\in\Omega$ and $V\in\C^N$.

Let $f:\D\to\Omega$ be holomorphic. By infinitesimal contraction,
\[
\bigl|d(H\circ f)_\zeta(\xi)\bigr|
\le C\kappa_\D(\zeta;\xi).
\]
Integrating along a Poincar\'e geodesic and then over chains of analytic discs gives
\begin{equation}
|H(z)-H(w)|\le Ck_\Omega(z,w),
\qquad z,w\in\Omega.
\label{eq:crit-H-lipschitz}
\end{equation}
For $z,w$ sufficiently close to the boundary, $H=h\circ(-r)$ at both endpoints, and \eqref{eq:crit-distance-pair} follows. Fixing $z_0\in\Omega$ gives \eqref{eq:crit-distance-point}.

Since
\[
\log\log\frac{A}{-r(z)}\longrightarrow+\infty
\qquad\text{as }z\to\partial\Omega,
\]
\eqref{eq:crit-distance-point} implies that every Kobayashi-bounded subset of $\Omega$ is contained in a set
\[
\{z\in\Omega:-r(z)\ge\varepsilon\}
\]
for some $\varepsilon>0$. This set is compact in $\Omega$. Hence every Kobayashi-Cauchy sequence has a Euclidean-convergent subsequence with limit in $\Omega$. If $z_{j_\nu}\to z\in\Omega$, choose a Euclidean ball $B$ centered at $z$ with $B\Subset\Omega$. For large $\nu$,
\[
k_\Omega(z_{j_\nu},z)\le k_B(z_{j_\nu},z)\longrightarrow0.
\]
Since the whole sequence is Kobayashi-Cauchy, it converges to $z$ in the Kobayashi distance. Thus $\Omega$ is Kobayashi complete.

If $A>\max_{\overline\Omega}(-r)$, the function
\[
H(z)=\log\log\frac{A}{-r(z)}
\]
is defined on all of $\Omega$. The same estimate \eqref{eq:crit-dH} holds globally, and \eqref{eq:crit-H-lipschitz} gives \eqref{eq:crit-distance-global}.
\end{proof}

\subsection{Self-improvement from a power estimate}

\subsubsection*{Finite-regularity hypotheses}

For the normal self-improvement argument, the relevant assumptions are on the pair $(\Omega,r)$ rather than on $\Omega$ alone. Let $\varepsilon>0$ and let $q$ be an integer satisfying
\[
q\varepsilon\ge1.
\]
These are sufficient regularity assumptions for the proof below; no optimality is asserted. In the case of a plurisubharmonic defining function it is enough to assume $q\ge2$. In the general pseudoconvex case, the mean-estimate argument below uses one additional derivative, so we assume $q\ge3$.

More precisely, let $\Omega\Subset\C^N$ be bounded and let $r\in C^q(U)$, where $U$ is a neighborhood of $\overline\Omega$, be a defining function:
\[
\Omega=\{r<0\},\qquad dr\neq0\quad\text{on }\partial\Omega.
\]
Assume that, for some $c_0>0$ and all $z$ sufficiently close to $\partial\Omega$,
\begin{equation}
\kappa_\Omega(z;X)\ge c_0\frac{\|X\|}{\delta_\Omega(z)^\varepsilon},
\qquad X\in\C^N.
\label{eq:crit-13}
\end{equation}
Since $-r\asymp\delta_\Omega$ near the boundary and $\Omega$ is bounded, after decreasing the constant we may use the global estimate
\begin{equation}
\kappa_\Omega(z;X)\ge c\frac{\|X\|}{(-r(z))^\varepsilon},
\qquad z\in\Omega,
\quad X\in\C^N.
\label{eq:crit-14}
\end{equation}
Indeed, on a compact subset of $\Omega$ the Kobayashi--Royden metric dominates a positive multiple of the Euclidean norm.

For a $C^2$ function $r$ we write
\[
L_r(z;X,Y)=\sum_{j,k=1}^N r_{z_j\bar z_k}(z)X_j\overline{Y_k},
\qquad
L_r(z;X)=L_r(z;X,X).
\]

The proof rests on the following elementary polynomial estimate.

\begin{lemma}[Bernstein estimate on a fixed interval]
For each $b>0$ there is a constant $C_b>0$ such that every complex polynomial $P$ of degree at most $d\ge1$ satisfies
\begin{equation}
|P'(0)|\le C_b d\sup_{t\in[-b,b]}|P(t)|.
\label{eq:crit-15}
\end{equation}
\end{lemma}

This is the classical Bernstein inequality for algebraic polynomials on an interval, after rescaling; see, for example,~\cite[Section~5.1]{BorweinErdelyi}.

\subsubsection{The case of a plurisubharmonic defining function}

\begin{proposition}[Normal self-improvement: plurisubharmonic defining function]\label{prop:self-psh}
Let $\Omega\Subset\C^N$ be bounded, let $r\in C^q(U)$ be a defining function as above, and assume
\[
q\ge2,
\qquad
q\varepsilon\ge1.
\]
Suppose that $r$ is plurisubharmonic on $U\cap\Omega$ and that the power estimate \eqref{eq:crit-13} holds. Then there are constants $c,A>0$ such that, for $z$ sufficiently close to $\partial\Omega$ and every $X\in\C^N$,
\begin{equation}
\kappa_\Omega(z;X)
\ge
c\,
\frac{\left|\displaystyle\sum_{j=1}^N r_{z_j}(z)X_j\right|}
{(-r(z))\log\!\dfrac{A}{-r(z)}}.
\label{eq:crit-16}
\end{equation}
\end{proposition}

\begin{proof}
Fix a holomorphic disc $f\in\mathcal O(\D,\Omega)$ and write
\[
z=f(0),
\qquad
g=-r\circ f>0,
\qquad
\delta=g(0)=-r(z).
\]
All constants below are uniform over $f$ and over sufficiently small $\delta$.

\medskip
\noindent\emph{Step 1: control of the analytic disc.}
Infinitesimal contraction and \eqref{eq:crit-14} give
\[
\frac{\|f'(\zeta)\|}{g(\zeta)^\varepsilon}
\le
\frac{C}{1-|\zeta|^2}.
\]
Hence, on a fixed smaller disc $\D_a$,
\begin{equation}
\|f'(\zeta)\|^{1/\varepsilon}\le Cg(\zeta).
\label{eq:crit-17}
\end{equation}
Since $r\circ f$ is subharmonic, $g=-r\circ f$ is positive and superharmonic. Hence
\begin{equation}
\frac1{2\pi}\int_0^{2\pi}g(te^{i\theta})\,d\theta\le g(0)=\delta,
\qquad 0<t<1.
\label{eq:crit-18}
\end{equation}
Integrating in $t$ gives
\[
\int_{\D_a}g\,dA\le C\delta.
\]
By \eqref{eq:crit-17},
\[
\int_{\D_a}\|f'(\zeta)\|^{1/\varepsilon}\,dA(\zeta)\le C\delta.
\]
The function $\|f'\|^{1/\varepsilon}$ is subharmonic. The submean inequality on fixed discs contained in $\D_a$ therefore yields
\[
\sup_{|\zeta|\le a/2}\|f'(\zeta)\|^{1/\varepsilon}\le C\delta,
\]
and hence
\begin{equation}
\sup_{|\zeta|\le a/2}
\bigl(\|f'(\zeta)\|+\|f(\zeta)-z\|\bigr)
\le C\delta^\varepsilon.
\label{eq:crit-19}
\end{equation}

\medskip
\noindent\emph{Step 2: logarithmic-degree polynomial approximation.}
Put $m=q-1$, so that
\[
(m+1)\varepsilon=q\varepsilon\ge1.
\]
Let $T_z$ be the real Taylor polynomial of $r$ at $z$ of degree $m$, written as a polynomial in the increment $h\in\C^N$, including the constant term $r(z)$. By Taylor's theorem and \eqref{eq:crit-19},
\begin{equation}
|r(f(\zeta))-T_z(f(\zeta)-z)|
\le C\delta^{(m+1)\varepsilon}
\le C\delta
\label{eq:crit-20}
\end{equation}
for $|\zeta|\le a/2$.

Let $P_L$ be the Taylor polynomial at $0$ of the holomorphic map $f-z$, of degree $L$. Cauchy's coefficient estimates applied on $\D_{a/2}$ give, on $\D_{a/4}$,
\begin{equation}
\|(f-z)-P_L\|\le C\delta^\varepsilon 2^{-L}.
\label{eq:crit-21}
\end{equation}
Choose
\[
L=\left\lceil C_0\log\frac{A}{\delta}\right\rceil
\]
with $C_0$ sufficiently large. Then the right-hand side of \eqref{eq:crit-21} is $o(\delta)$. Since the first derivatives of $T_z$ are uniformly bounded, the real polynomial
\[
Q(\zeta,\bar\zeta)=T_z(P_L(\zeta))
\]
satisfies
\begin{equation}
\sup_{|\zeta|\le a/4}|Q(\zeta)-r(f(\zeta))|\le C\delta,
\qquad
\deg Q\le mL\le C\log\frac{A}{\delta}.
\label{eq:crit-22}
\end{equation}
Moreover, the derivative at the center is preserved exactly:
\begin{equation}
\partial_\zeta Q(0)
=
\sum_{j=1}^N r_{z_j}(z)f_j'(0).
\label{eq:crit-23}
\end{equation}

\medskip
\noindent\emph{Step 3: extract the complex-normal coefficient.}
For real $t$ with $|t|\le a/4$, define
\[
B(t)=\frac1{2\pi}\int_0^{2\pi}Q(te^{i\theta},te^{-i\theta})e^{-i\theta}\,d\theta.
\]
If
\[
Q(\zeta,\bar\zeta)=\sum_{j,k}c_{jk}\zeta^j\bar\zeta^k,
\]
then
\[
B(t)=\sum_{j-k=1}c_{jk}t^{j+k}.
\]
Thus $B$ is a complex polynomial with
\[
\deg B\le\deg Q,
\qquad
B'(0)=c_{10}=\partial_\zeta Q(0).
\]
Since $r\circ f=-g$, estimates \eqref{eq:crit-18} and \eqref{eq:crit-22} imply
\[
|B(t)|
\le
\frac1{2\pi}\int_0^{2\pi}|Q(te^{i\theta})|\,d\theta
\le
\frac1{2\pi}\int_0^{2\pi}g(te^{i\theta})\,d\theta+C\delta
\le C\delta.
\]
Applying the Bernstein estimate on the fixed interval $[-a/4,a/4]$ gives
\[
|B'(0)|
\le
C\,\deg(B)\,\delta
\le
C\delta\log\frac{A}{\delta}.
\]
Using \eqref{eq:crit-23}, we obtain
\begin{equation}
\left|
\sum_{j=1}^N r_{z_j}(z)f_j'(0)
\right|
\le
C(-r(z))\log\frac{A}{-r(z)}.
\label{eq:crit-24}
\end{equation}

\medskip
\noindent\emph{Step 4: pass to the Kobayashi metric.}
If $f$ is admissible in the definition of $\kappa_\Omega(z;X)$ and
\[
f'(0)=\lambda^{-1}X,
\]
then \eqref{eq:crit-24} gives
\[
\lambda^{-1}
\left|\sum_{j=1}^N r_{z_j}(z)X_j\right|
\le
C(-r(z))\log\frac{A}{-r(z)}.
\]
Hence
\[
\lambda
\ge
c\,
\frac{\left|\displaystyle\sum_{j=1}^N r_{z_j}(z)X_j\right|}
{(-r(z))\log\!\dfrac{A}{-r(z)}}.
\]
Taking the infimum over all admissible discs proves \eqref{eq:crit-16}.
\end{proof}

\subsubsection{The general pseudoconvex case}

The preceding proof uses plurisubharmonicity only once: to obtain the superharmonic mean estimate \eqref{eq:crit-18}. The rest of the argument depends only on a uniform mean estimate for $g=-r\circ f$. We now recover exactly such an estimate from pseudoconvexity.

\begin{lemma}[Uniform disc mean estimate]\label{lem:disc-mean}
Let $\Omega\Subset\C^N$ be pseudoconvex and let $r\in C^3(U)$ be a defining function. Then there are constants $a\in(0,1/4)$ and $C>0$ such that for every $f\in\mathcal O(\D,\Omega)$, if $g=-r\circ f$, then
\begin{equation}
\frac1{2\pi}\int_0^{2\pi}g(te^{i\theta})\,d\theta
\le Cg(0),
\qquad 0<t\le a.
\label{eq:crit-25}
\end{equation}
\end{lemma}

\begin{proof}
We first prove the boundary Levi estimate. Near $\partial\Omega$ put
\[
\nu(z)=
\frac{(r_{\bar z_1}(z),\ldots,r_{\bar z_N}(z))}
{\sum_{j=1}^N|r_{z_j}(z)|^2},
\qquad
\partial r(z)\nu(z)=1.
\]
For $X\in\C^N$ set
\[
E(z;X)=
L_r(z;\nu(z),X)
-
\frac12\overline{\partial r(z)X}\,L_r(z;\nu(z),\nu(z)).
\]
Then $E(z;\cdot)$ is conjugate-linear. Pseudoconvexity on the boundary, together with the $C^3$ regularity of $r$, gives
\begin{equation}
L_r(z;X)
\ge
-C(-r(z))\|X\|^2
+2\operatorname{Re}\bigl((\partial r(z)X)E(z;X)\bigr)
\label{eq:crit-26}
\end{equation}
for $z$ near the boundary. The coefficients of $E$ and their first derivatives are bounded. Extend the coefficients of $E$ to $\Omega$ with bounded first derivatives. After enlarging $C$, the same inequality holds on all of $\Omega$.

For completeness, the estimate follows by writing
\[
\eta=\partial r(z)X,
\qquad
Y=X-\eta\nu(z),
\]
so that $\partial r(z)Y=0$. If $p$ is the nearest boundary point to $z$, then $\|z-p\|\le C(-r(z))$. Projecting $Y$ to the complex tangent space of $\partial\Omega$ at $p$ changes it by $O((-r(z))\|Y\|)$. Since the Levi form is nonnegative on that tangent space and its coefficients are Lipschitz,
\[
L_r(z;Y)\ge-C(-r(z))\|Y\|^2.
\]
The exact identity
\[
L_r(z;X)=L_r(z;Y)+2\operatorname{Re}\bigl(\eta E(z;X)\bigr)
\]
then gives \eqref{eq:crit-26}.

Now let $f\in\mathcal O(\D,\Omega)$ and $g=-r\circ f>0$. Since $f$ is holomorphic,
\[
g_\zeta=-\partial r(f)f',
\qquad
\Delta g=-4L_r(f;f').
\]
Applying \eqref{eq:crit-26} with $X=f'(\zeta)$ gives
\begin{equation}
\Delta g
\le
4Cg\|f'\|^2+8\operatorname{Re}\bigl(g_\zeta E(f;f')\bigr).
\label{eq:crit-27}
\end{equation}
Write
\[
E(f(\zeta);f'(\zeta))=\alpha(\zeta)+i\beta(\zeta).
\]
Since $g$ is real-valued,
\[
8\operatorname{Re}\bigl(g_\zeta E(f;f')\bigr)
=4\alpha g_x+4\beta g_y.
\]
Because $\Omega$ is bounded, Cauchy's estimates give uniform bounds for $f'$ and $f''$ on $\D_{3/4}$. The coefficients of $E$ and their first derivatives are bounded, hence $\alpha$ and $\beta$ have uniformly bounded $C^1$ norms there. Thus, on $\D_{3/4}$,
\begin{equation}
\Delta g\le b_1g_x+b_2g_y+c_0g,
\label{eq:crit-28}
\end{equation}
where the $C^1$ norms of $b_1,b_2$ and the $L^\infty$ norm of $c_0$ are bounded independently of $f$.

It remains to use a one-variable barrier. If a nonnegative $C^2$ function $u$ on $\D$ satisfies
\[
\Delta u\le b_1u_x+b_2u_y+cu
\]
with the same uniform coefficient bounds, then there are $K>0$ and $a>0$, depending only on those bounds, such that
\begin{equation}
\frac1{2\pi}\int_0^{2\pi}u(te^{i\theta})\,d\theta
\le
\frac{u(0)}{1-Kt},
\qquad 0<t\le a.
\label{eq:crit-29}
\end{equation}
Indeed, for fixed $t$ put
\[
\Psi_t(\zeta)=\log\frac{t}{|\zeta|}-K(t-|\zeta|),
\qquad 0<|\zeta|<t.
\]
The formal adjoint of
\[
L=\Delta-b_1\partial_x-b_2\partial_y-c
\]
is
\[
L^*=\Delta+b_1\partial_x+b_2\partial_y
+\bigl((b_1)_x+(b_2)_y-c\bigr).
\]
On the punctured disc put
\[
F_t=\Delta\Psi_t
+b_1(\Psi_t)_x+b_2(\Psi_t)_y
+\bigl((b_1)_x+(b_2)_y-c\bigr)\Psi_t.
\]
If $s=|\zeta|$ and $Kt\le1/2$, then
\[
\Delta\Psi_t=\frac{K}{s},
\qquad
\bigl((\Psi_t)_x^2+(\Psi_t)_y^2\bigr)^{1/2}\le\frac1s,
\qquad
0\le\Psi_t\le\log\frac ts\le\frac ts.
\]
Thus, if the $C^1$ norms of $b_1,b_2$ and the $L^\infty$ norm of $c$ are at most $M$,
\[
F_t(\zeta)\ge\frac{K-CM-CMt}{|\zeta|}.
\]
Choose $K$ sufficiently large in terms of $M$, and then choose $a>0$ sufficiently small. We then have $F_t\ge0$ for $0<t\le a$. Since
\[
\Delta\log\frac t{|\zeta|}=-2\pi\delta_0
\]
in the sense of distributions,
\[
L^*\Psi_t=-2\pi\delta_0+F_t.
\]
Green's identity on $\{\epsilon<|\zeta|<t\}$, followed by $\epsilon\downarrow0$, gives
\[
-\int_{\partial\D_t}u\,\partial_\nu\Psi_t\,ds
\le2\pi u(0).
\]
Since
\[
-\partial_\nu\Psi_t=\frac1t-K
\qquad\text{on }\partial\D_t,
\]
we obtain \eqref{eq:crit-29}. After decreasing $a$, we may assume $Ka\le1/2$, and hence the right-hand side of \eqref{eq:crit-29} is at most $2u(0)$.

Apply this statement to a fixed rescaling of $g$ from $\D_{3/4}$ to $\D$. This gives \eqref{eq:crit-25}.
\end{proof}

\begin{proposition}[Normal self-improvement: general pseudoconvex case]\label{prop:self-general}
Let $\Omega\Subset\C^N$ be a bounded pseudoconvex domain and let $r\in C^q(U)$ be a defining function, where
\[
q\ge3,
\qquad
q\varepsilon\ge1.
\]
Assume the power estimate \eqref{eq:crit-13}. Then there are constants $c,A>0$ such that the normal estimate \eqref{eq:crit-16} holds for $z$ sufficiently close to $\partial\Omega$ and every $X\in\C^N$.
\end{proposition}

\begin{proof}
Let $f\in\mathcal O(\D,\Omega)$, put $g=-r\circ f$, and let $\delta=g(0)$. The uniform disc mean estimate gives
\[
\frac1{2\pi}\int_0^{2\pi}g(te^{i\theta})\,d\theta\le C\delta
\]
on a fixed parameter disc. This is the only place where the proof of the plurisubharmonic case used the assumption that $r$ itself was plurisubharmonic.

Consequently, the argument proving \eqref{eq:crit-19} is unchanged: on a fixed smaller parameter disc, the power estimate shows that the image of $f$ lies in a Euclidean ball of radius $O(\delta^\varepsilon)$ around $f(0)$. Since $q\varepsilon\ge1$, the Taylor remainder of order $q-1$ is $O(\delta)$. This is sufficient: the logarithmic-degree approximation has error $O(\delta)$, and the extraction of the first Fourier coefficient together with the Bernstein estimate proceeds exactly as in the preceding proposition. They again give
\begin{equation}\label{eq:crit-general-disc}
\left|
\sum_{j=1}^N r_{z_j}(z)f_j'(0)
\right|
\le
C(-r(z))\log\frac{A}{-r(z)}.
\end{equation}
Passing to admissible discs in the definition of the Kobayashi--Royden metric proves \eqref{eq:crit-16}.
\end{proof}

\begin{corollary}
Under the hypotheses of either of the two preceding propositions, $\Omega$ is Kobayashi complete. More precisely, for every fixed $z_0\in\Omega$,
\[
k_\Omega(z_0,z)
\ge
c\log\log\frac{A}{-r(z)}-C_{z_0}
\]
for $z$ sufficiently close to $\partial\Omega$.
\end{corollary}

\begin{proof}
The normal estimate \eqref{eq:crit-16} is exactly the hypothesis of Proposition~\ref{prop:criterion}.
\end{proof}

\begin{remark}[A sharper normal estimate]\label{rem:sharp-normal}
Assume now that $\Omega\Subset\C^N$ is a bounded pseudoconvex domain with
$C^\infty$ boundary, let $r$ be a $C^\infty$ defining function on a neighborhood of $\overline\Omega$, and suppose that for some
$c,\varepsilon>0$,
\[
 \kappa_\Omega(z;X)\ge c\frac{\|X\|}{(-r(z))^\varepsilon},
 \qquad z\in\Omega,\quad X\in\C^N.
\]
The preceding argument already proves completeness.  With the same idea one
can improve the normal estimate further: for sufficiently large $A$,
\begin{equation}\label{eq:sharp-normal}
 \kappa_\Omega(z;X)
 \ge
 c\frac{|\partial r(z)X|}{-r(z)}
 \frac{\log\log(A/(-r(z)))}{\log(A/(-r(z)))}
\end{equation}
for $z$ sufficiently close to $\partial\Omega$.  Consequently, for every
fixed $z_0\in\Omega$,
\begin{equation}\label{eq:sharp-distance}
 k_\Omega(z_0,z)
 \ge
 c\left(\log\log\frac{A}{-r(z)}\right)^2-C_{z_0}
\end{equation}
near the boundary.  This refinement is not needed for completeness.

\medskip
\noindent\emph{Proof.}
We first use the following one-variable estimate.  There is an absolute
constant $C>0$ such that every real polynomial $P(\zeta,\bar\zeta)$ of
total degree at most $d\ge1$ satisfying
\[
 P\ge0,\qquad \Delta P\le0\quad\text{on }\D
\]
satisfies
\begin{equation}\label{eq:super-poly-remark}
 \bigl(|P_x(0)|^2+|P_y(0)|^2\bigr)^{1/2}
 \le C\frac{d+1}{\log(d+2)}P(0).
\end{equation}
To prove this, put $\mu=-\Delta P\ge0$, and let $H$ be the harmonic
function in $\D$ with boundary values $P$.  Green's formula gives
\begin{equation}\label{eq:poly-green-remark}
 P(z)=H(z)+\frac1{2\pi}\int_\D
 \log\left|\frac{1-z\bar w}{z-w}\right|\mu(w)\,dA(w).
\end{equation}
At $z=0$ all terms on the right are nonnegative, and hence
\begin{equation}\label{eq:poly-mass-remark}
 H(0)\le P(0),\qquad
 \frac1{2\pi}\int_\D\log\frac1{|w|}\mu(w)\,dA(w)\le P(0).
\end{equation}
Differentiating \eqref{eq:poly-green-remark} at the origin and using the
standard first-derivative estimate for harmonic functions gives
\begin{equation}\label{eq:poly-grad-remark}
 \bigl(|P_x(0)|^2+|P_y(0)|^2\bigr)^{1/2}
 \le CP(0)+\frac1{2\pi}\int_\D
 \left(\frac1{|w|}-|w|\right)\mu(w)\,dA(w).
\end{equation}
Fix $R=1/4$ and write
\[
 M(t)=\frac1{2\pi}\int_0^{2\pi}\mu(te^{i\theta})\,d\theta,
 \qquad I=\int_0^R M(t)\,dt.
\]
The function $M$ is an even nonnegative real polynomial of degree at most
$d$.  Consider
\[
 S(\theta)=M(R\cos\theta)\sin^2\theta.
\]
This is a nonnegative trigonometric polynomial of degree at most $d+2$.
If $S(\theta)=\sum_{|j|\le D}c_je^{ij\theta}$, then nonnegativity gives
$|c_j|\le c_0$, where $c_0$ is the mean of $S$.  Hence
\[
 \|S\|_\infty\le C(d+1)c_0.
\]
For $0\le t\le R/2$, choose $\theta$ with $t=R\cos\theta$.  Since
$\sin^2\theta\ge3/4$, a change of variables gives
\begin{equation}\label{eq:M-sup-remark}
 \sup_{0\le t\le R/2}M(t)\le C(d+1)I.
\end{equation}
Choose $h=c/(d+1)<R/2$ with $c>0$ sufficiently small.  Then
\eqref{eq:M-sup-remark} implies
\[
 \int_0^hM(t)\,dt\le\frac I2,
 \qquad
 \int_h^R M(t)\,dt\ge\frac I2.
\]
Since $t\mapsto t\log(1/t)$ is increasing on $[0,R]$,
\[
 \int_0^R t\log(1/t)M(t)\,dt
 \ge c\frac{\log(d+2)}{d+1}I.
\]
On the other hand, polar coordinates in \eqref{eq:poly-mass-remark} give
\[
 \int_0^1 t\log(1/t)M(t)\,dt\le P(0).
\]
Therefore
\begin{equation}\label{eq:I-bound-remark}
 I\le C\frac{d+1}{\log(d+2)}P(0).
\end{equation}
The part of the integral in \eqref{eq:poly-grad-remark} over $|w|<R$ is
at most $I$, while on $R\le|w|<1$ one has
\[
 |w|^{-1}-|w|\le C\log\frac1{|w|}.
\]
Using \eqref{eq:poly-mass-remark} and \eqref{eq:I-bound-remark} in
\eqref{eq:poly-grad-remark} proves \eqref{eq:super-poly-remark}.

We now return to an analytic disc $f\in\mathcal O(\D,\Omega)$.  Put
\[
 z=f(0),\qquad g=-r\circ f>0,\qquad \delta=g(0)=-r(z).
\]
Decreasing $\varepsilon$ if necessary, we may assume
$0<\varepsilon\le1/2$.  The proof of
Proposition~\ref{prop:self-general} gives, on a fixed parameter disc,
\begin{equation}\label{eq:sharp-disc-shrink}
 \|f-z\|+\|f'\|+\|f''\|\le C\delta^\varepsilon.
\end{equation}
Choose an integer $m$ with $(m+1)\varepsilon>2$.  Let $T_z$ be the real
Taylor polynomial of $r$ at $z$ of degree $m$, viewed as a polynomial in the ambient variable, and let $P_L$ be the
Taylor polynomial of $f$ at $0$ of degree
\[
 L=\left\lceil C_0\log\frac A\delta\right\rceil,
\]
where $C_0$ is sufficiently large.  On a fixed smaller parameter disc,
Taylor's theorem, \eqref{eq:sharp-disc-shrink}, and Cauchy's estimates
give a real polynomial
\[
 Q_0=T_z\circ P_L
\]
such that
\begin{equation}\label{eq:sharp-Q0}
 \|Q_0-r\circ f\|_{C^0}\le C\delta^2,
 \qquad
 \deg Q_0\le C\log\frac A\delta,
 \qquad
 \partial_\zeta Q_0(0)=\partial r(z)f'(0).
\end{equation}
Write
\[
 T=\log\frac A\delta.
\]
For fixed $0\le t\le a/4$, where $a$ is the radius in
Lemma~\ref{lem:disc-mean}, choose $C_1$ sufficiently large and put
\[
 \Theta_t(\theta)=-Q_0(te^{i\theta})+C_1\delta^2.
\]
By \eqref{eq:sharp-Q0}, $\Theta_t$ is a nonnegative trigonometric
polynomial of degree at most $CT$.  Its mean is at most $C\delta$ by
Lemma~\ref{lem:disc-mean}.  For a nonnegative trigonometric polynomial,
every Fourier coefficient is bounded by its mean; hence
\[
 \sup_\theta\Theta_t(\theta)\le C\delta T.
\]
Again using \eqref{eq:sharp-Q0}, we obtain
\begin{equation}\label{eq:sharp-depth}
 \sup_{|\zeta|\le a/4}g(\zeta)\le C\delta T.
\end{equation}

Choose a fixed $b\in(0,a/8)$.  By \eqref{eq:sharp-depth}, all centers
$f(\zeta)$ with $|\zeta|\le2b$ are sufficiently close to the boundary
when $\delta$ is small.  Apply the already proved disc estimate
\eqref{eq:crit-general-disc} to $f\circ\alpha_\zeta$, where $\alpha_\zeta$ is a
disc automorphism satisfying
\[
 \alpha_\zeta(0)=\zeta,
 \qquad |\alpha_\zeta'(0)|=1-|\zeta|^2.
\]
Since
\[
 g_\zeta=-\partial r(f)f',
\]
we obtain
\[
 \bigl(|g_x(\zeta)|^2+|g_y(\zeta)|^2\bigr)^{1/2}
 \le Cg(\zeta)\log\frac A{g(\zeta)}.
\]
The function $x\mapsto x\log(A/x)$ is increasing for small $x$.
Together with \eqref{eq:sharp-depth}, this gives
\begin{equation}\label{eq:sharp-gradient}
 \bigl(|g_x(\zeta)|^2+|g_y(\zeta)|^2\bigr)^{1/2}
 \le C\delta T^2,
 \qquad |\zeta|\le2b.
\end{equation}
Using the Levi estimate \eqref{eq:crit-26}, the identity
$\Delta g=-4L_r(f;f')$, and \eqref{eq:sharp-disc-shrink},
\eqref{eq:sharp-depth}, \eqref{eq:sharp-gradient}, we get
\[
 \Delta g\le C\delta^{1+\varepsilon}T^2
 \qquad\text{on }\D_{2b}.
\]
Set $\eta=\varepsilon/2$.  Since
$\delta^{\varepsilon/2}T^2$ is bounded as $\delta\downarrow0$,
\begin{equation}\label{eq:sharp-laplacian}
 \Delta g\le C\delta^{1+\eta}
 \qquad\text{on }\D_{2b}.
\end{equation}

Increase $m$, if necessary, so that $(m+1)\varepsilon>3$.  Repeating the
polynomial approximation, now in the $C^2$ norm on $\D_b$, gives a real
polynomial
\[
 Q=-T_z\circ P_L
\]
with
\begin{equation}\label{eq:sharp-C2}
\begin{gathered}
 \|Q-g\|_{C^2(\D_b)}\le C\delta^2,
 \qquad
 Q(0)=\delta,
 \qquad
 \deg Q\le CT,\\
 Q_x(0)=g_x(0),
 \qquad
 Q_y(0)=g_y(0).
\end{gathered}
\end{equation}
Indeed, if $R_m=r-T_z$, then for $j=0,1,2$,
\[
 |D^jR_m(z+h)|\le C\|h\|^{m+1-j},
\]
and \eqref{eq:sharp-disc-shrink} shows that
$\|D^j(R_m\circ f)\|_{C^0(\D_b)}=O(\delta^{(m+1)\varepsilon})$.
The first two derivatives of the Taylor tail $f-P_L$ are bounded by
\[
 C\delta^\varepsilon(L+1)^j2^{-L},\qquad j=0,1,2,
\]
which is $O(\delta^3)$ for $C_0$ sufficiently large.  This proves
\eqref{eq:sharp-C2}.

Choose a fixed $B>0$ sufficiently large and define on $\D_b$
\[
 P(\zeta)=Q(\zeta)+B\delta^2
 +B\delta^{1+\eta}(b^2-|\zeta|^2).
\]
By \eqref{eq:sharp-C2}, $P\ge0$ on $\D_b$.  By
\eqref{eq:sharp-laplacian} and \eqref{eq:sharp-C2},
\[
 \Delta P
 \le C\delta^{1+\eta}+C\delta^2-4B\delta^{1+\eta}\le0
\]
for all sufficiently small $\delta$, after increasing $B$.  Moreover,
\[
 P(0)\le2\delta,
 \qquad
 P_x(0)=g_x(0),\qquad P_y(0)=g_y(0),
 \qquad
 \deg P\le CT.
\]
Rescale $\D_b$ to the unit disc and apply
\eqref{eq:super-poly-remark}.  Since
$\log(CT+2)\asymp\log T$ for large $T$, we obtain
\[
 \bigl(|g_x(0)|^2+|g_y(0)|^2\bigr)^{1/2}
 \le C\delta\frac{T}{\log T}.
\]
Since
\[
\bigl(|g_x(0)|^2+|g_y(0)|^2\bigr)^{1/2}
=2|g_\zeta(0)|=2|\partial r(z)f'(0)|,
\]
this gives
\begin{equation}\label{eq:sharp-disc-normal}
 |\partial r(z)f'(0)|
 \le C\delta
 \frac{\log(A/\delta)}{\log\log(A/\delta)}.
\end{equation}
Passing to admissible discs in the definition of the Kobayashi--Royden
metric proves \eqref{eq:sharp-normal}.

To obtain the distance estimate, define near the boundary
\[
 H(z)=\frac12\left(\log\log\frac{A}{-r(z)}\right)^2.
\]
For a real tangent vector $V$,
\[
 |dH_z(V)|
 \le
 C\frac{\log\log(A/(-r(z)))}{(-r(z))\log(A/(-r(z)))}
 |\partial r(z)V|.
\]
Hence \eqref{eq:sharp-normal} gives $|dH_z(V)|\le C\kappa_\Omega(z;V)$ near the boundary. Extend $H$ to a $C^1$ function on the interior as in the proof of Proposition~\ref{prop:criterion}; on the remaining compact set its differential is bounded and $\kappa_\Omega$ dominates the Euclidean norm. Integrating along analytic-disc chains therefore gives
\[
 |H(z)-H(z_0)|\le Ck_\Omega(z_0,z),
\]
which is \eqref{eq:sharp-distance}. \hfill$\square$
\end{remark}

\subsection{Real-analytic boundary}
\begin{proof}[Proof of Corollary~\ref{cor:real-analytic-intro}]
Diederich and Forn\ae ss proved that there are constants $c,s>0$ such that, for $z$ sufficiently close to $\partial\Omega$,
\[
 \kappa_\Omega(z;X)\ge c\frac{\|X\|}{\delta_\Omega(z)^s},
 \qquad X\in\C^N;
\]
see~\cite[Theorem~4, p.~577]{DiederichFornaess79}.  Since a real-analytic boundary is $C^\infty$, choose a $C^\infty$ defining function $r$, and then choose an integer $q\ge3$ so large that $qs\ge1$.  Proposition~\ref{prop:self-general} applies with $\varepsilon=s$ and gives
\[
 \kappa_\Omega(z;X)
 \ge
 c\,
 \frac{|\partial r(z)X|}
 {(-r(z))\log(A/(-r(z)))}
\]
near the boundary.  Proposition~\ref{prop:criterion} then implies that $\Omega$ is Kobayashi complete.
\end{proof}

\section{Finite type: a power lower bound for the Kobayashi metric}\label{sec:finite-type-power}

The local estimate below is the estimate that was sought in Cho's paper~\cite{Cho}.  We do not use Cho's argument.  The proof given here is self-contained once the uniform disc mean estimate above and D'Angelo's type--multiplicity inequality are available.

The algebraic ingredient is D'Angelo's type--multiplicity inequality.  Let
$\mathcal O_0$ denote the ring of holomorphic germs at $0\in\C^N$.
For a proper ideal $I\subset\mathcal O_0$, define
\[
 \Delta(I)=\sup_\gamma
 \frac{\min_{h\in I}\operatorname{ord}_0(h\circ\gamma)}
 {\operatorname{ord}_0\gamma},
\]
where $\gamma$ ranges over nonconstant holomorphic curve germs through
the origin.  If $\Delta(I)<\infty$, then
\begin{equation}\label{eq:ideal-multiplicity}
 \dim_\C\mathcal O_0/I\le\Delta(I)^N.
\end{equation}
This is D'Angelo's type--multiplicity inequality
\cite[Theorem~2.7]{DAngelo}.  We shall use it only in the following
form: once we construct an ideal whose type is at most the boundary type
$T$, the quotient algebra has uniformly bounded finite dimension.

\begin{proposition}\label{prop:finite-type-power}
Suppose that $\Omega\Subset\C^N$ has $C^\infty$ pseudoconvex
boundary and that $p\in\partial\Omega$ has finite D'Angelo
one-type. Then there exist a neighborhood $V$ of $p$ and
constants $c,\varepsilon>0$ such that
\begin{equation}\label{eq:weight-power}
 \kappa_\Omega(z;X)\ge
 c\frac{\|X\|}{\delta_\Omega(z)^\varepsilon},
 \qquad z\in\Omega\cap V,\quad X\in\C^N.
\end{equation}
\end{proposition}

Before proving the proposition, we state separately the elementary finite-jet
facts used in the argument.  The first converts the mean estimate from
the preceding subsection into quantitative control of Taylor
coefficients.

\begin{lemma}\label{lem:finite-jet-control}
Fix $a>0$ and integers $q\ge1$ and $0\le \ell\le q$.  For every
$M,C_0>0$ there are constants $C,\delta_0>0$ with the following
property.  Suppose that $0<\delta\le\delta_0$ and
$u\in C^{q+1}(\D_a)$ satisfies
\[
 u\ge0,\qquad \|u\|_{C^{q+1}(\D_a)}\le M,
\]
and
\begin{equation}\label{eq:finite-jet-mean}
 \int_{\D_\rho}u\,dA\le C_0\delta\rho^2,
 \qquad 0<\rho\le a.
\end{equation}
If
\[
 Q(\zeta,\bar\zeta)=\sum_{j+k\le q}c_{jk}\zeta^j\bar\zeta^k
\]
is the Taylor polynomial of $u$ at $0$ of total degree $q$, then
\begin{equation}\label{eq:finite-jet-coeff}
 |c_{jk}|\le C\delta^{1-\frac{j+k}{q+1}},
 \qquad j+k\le q.
\end{equation}
In particular,
\begin{equation}\label{eq:finite-jet-low}
 |c_{jk}|\le C\delta^{1-\frac{\ell}{q+1}},
 \qquad j+k\le\ell.
\end{equation}
\end{lemma}

\begin{proof}
Put
\[
 \rho=\delta^{1/(q+1)}.
\]
After decreasing $\delta_0$, we may assume $\rho\le a$.  Taylor's
theorem and the $C^{q+1}$ bound give, for $|\zeta|\le1$,
\[
 |Q(\rho\zeta)-u(\rho\zeta)|\le C\rho^{q+1}=C\delta.
\]
By \eqref{eq:finite-jet-mean},
\[
 \int_\D u(\rho\zeta)\,dA(\zeta)
 =\rho^{-2}\int_{\D_\rho}u(w)\,dA(w)
 \le C_0\delta.
\]
Since $u\ge0$,
\[
 \int_\D |Q(\rho\zeta)|\,dA(\zeta)\le C\delta.
\]
On the finite-dimensional space of polynomials in $\zeta,\bar\zeta$
of total degree at most $q$, the coefficient norm is bounded by a
constant times the $L^1(\D)$ norm.  Applied to
\[
 Q(\rho\zeta)=\sum_{j+k\le q}c_{jk}\rho^{j+k}
 \zeta^j\bar\zeta^k,
\]
this gives
\[
 |c_{jk}|\rho^{j+k}\le C\delta.
\]
Substituting the choice of $\rho$ proves
\eqref{eq:finite-jet-coeff}; \eqref{eq:finite-jet-low} follows because
$0<\delta\le1$.
\end{proof}

The next lemma is the finite-dimensional polarization that will allow us
to pass from a real Taylor polynomial to holomorphic data.

\begin{lemma}\label{lem:hermitian-decomposition}
Let $K$ be compact and let $R_z(v,\bar v)$, $z\in K$, be a continuous
family of real-valued polynomials on $\C^N$ of degree at most $d\ge2$,
with $R_z(0)=0$.  Then there exist a finite-dimensional Hermitian space
$H$ and holomorphic polynomials
\[
 h_z:\C^N\to\C,
 \qquad
 F_z,G_z:\C^N\to H,
\]
all vanishing at $0$, such that
\begin{equation}\label{eq:hermitian-decomposition}
 R_z(v,\bar v)=2\operatorname{Re} h_z(v)+\|F_z(v)\|^2-\|G_z(v)\|^2.
\end{equation}
The coefficients of $h_z,F_z,G_z$ may be chosen to depend continuously
on $z$.  If, in addition, $K$ is contained in a smooth parameter
manifold and the coefficients of $R_z$ extend smoothly to a
neighborhood of $K$, then the same is true of $h_z,F_z,G_z$.  Moreover,
$H$ may be chosen independently of $z$, and $G_z$ may be chosen
independent of $z$.
\end{lemma}

\begin{proof}
Write
\[
 R_z(v,\bar v)
 =\sum_{1\le |\alpha|+|\beta|\le d}
 c_{\alpha\beta}(z)v^\alpha\bar v^\beta.
\]
Since $R_z$ is real-valued,
\[
 c_{\alpha\beta}(z)=\overline{c_{\beta\alpha}(z)}.
\]
Collect the pure holomorphic terms into
\[
 h_z(v)=\sum_{1\le|\alpha|\le d}c_{\alpha0}(z)v^\alpha.
\]
The pure antiholomorphic terms are then $\overline{h_z(v)}$, so the
remaining part of $R_z$ consists only of mixed monomials.

Let $V(v)$ be the column vector of all monomials $v^\alpha$ with
$1\le|\alpha|\le d-1$.  The mixed part can be written
\[
 V(v)^*A(z)V(v),
\]
where $A(z)$ is Hermitian and depends continuously on $z$; entries
which would produce total degree greater than $d$ are set equal to
zero.  Choose
\[
 L>\sup_{z\in K}\|A(z)\|.
\]
Then $A(z)+LI$ is positive definite for every $z\in K$.  Set
\[
 F_z(v)=(A(z)+LI)^{1/2}V(v),
 \qquad
 G_z(v)=\sqrt L\,V(v).
\]
Then
\[
 \|F_z(v)\|^2-\|G_z(v)\|^2=V(v)^*A(z)V(v),
\]
which proves \eqref{eq:hermitian-decomposition}.  The positive
square-root map is continuous, and smooth on the cone of positive
definite Hermitian matrices.  This gives the asserted continuous
dependence, and also the asserted smooth dependence when a smooth
parameter structure is present.
The individual squared norms may contain terms of degree greater than
$d$, but those terms arise from the added $LI$ and cancel exactly in
their difference.
\end{proof}

We also state the precise consequence of finite D'Angelo type for the
unitary family of ideals associated with such a decomposition.

\begin{lemma}\label{lem:unitary-colength}
Let $r$ be a smooth real-valued germ at $0\in\C^N$, with $r(0)=0$, and
suppose that for every nonconstant holomorphic curve germ $\gamma$
through $0$,
\begin{equation}\label{eq:type-bound-germ}
 \operatorname{ord}_0(r\circ\gamma)
 \le T\operatorname{ord}_0\gamma
\end{equation}
for some finite $T$.  Let $R$ be the Taylor polynomial of $r$ of degree
$d>T$, and suppose that
\[
 R=2\operatorname{Re} h+\|F\|^2-\|G\|^2
\]
with $h,F,G$ holomorphic polynomials vanishing at $0$, the last two
with values in the same finite-dimensional Hermitian space.  For every
unitary $U$ on that space, put
\[
 I_U=(h,F-UG)\subset\mathcal O_0.
\]
Then
\begin{equation}\label{eq:unitary-type}
 \Delta(I_U)\le T,
\end{equation}
and consequently
\begin{equation}\label{eq:unitary-colength}
 \dim_\C\mathcal O_0/I_U\le T^N.
\end{equation}
\end{lemma}

\begin{proof}
Fix a nonconstant curve germ $\gamma$ and put
$q=\operatorname{ord}_0\gamma$.  The Taylor remainder $r-R$ has order
at least $d+1$ in the ambient variables, hence
\[
 \operatorname{ord}_0((r-R)\circ\gamma)\ge(d+1)q>Tq.
\]
Together with \eqref{eq:type-bound-germ}, this gives
\begin{equation}\label{eq:R-same-order}
 \operatorname{ord}_0(R\circ\gamma)
 =\operatorname{ord}_0(r\circ\gamma)\le Tq.
\end{equation}
Since $U$ is unitary,
\[
 \|F\|^2-\|G\|^2
 =(F-UG)\cdot\overline F
 +(UG)\cdot\overline{(F-UG)}.
\]
Let
\[
 s_\gamma=
 \min\Bigl\{
 \operatorname{ord}_0(h\circ\gamma),
 \operatorname{ord}_0((F-UG)_\nu\circ\gamma):\nu
 \Bigr\}.
\]
This is exactly
\[
 s_\gamma=\min_{a\in I_U}\operatorname{ord}_0(a\circ\gamma),
\]
because the displayed functions generate $I_U$.  Every summand of
$R\circ\gamma$ has order at least $s_\gamma$; cancellation can only
increase the order.  Hence, by \eqref{eq:R-same-order},
\[
 s_\gamma\le\operatorname{ord}_0(R\circ\gamma)\le Tq.
\]
Taking the supremum over $\gamma$ proves \eqref{eq:unitary-type}.
D'Angelo's inequality \eqref{eq:ideal-multiplicity} then gives
\eqref{eq:unitary-colength}.
\end{proof}

The fourth lemma says that two finite lists with almost identical Gram
matrices differ, after one unitary change of coordinates, by the square
root of the Gram-matrix error.

\begin{lemma}\label{lem:gram-matching}
Let $A_1,\ldots,A_m$ and $B_1,\ldots,B_m$ lie in a finite-dimensional
Hermitian space $H$ with $\dim H\ge m$.  If
\[
 |\langle A_j,A_k\rangle-\langle B_j,B_k\rangle|\le\eta,
 \qquad 1\le j,k\le m,
\]
then there is a unitary operator $U$ on $H$ such that
\begin{equation}\label{eq:gram-matching}
 \|A_j-UB_j\|\le\sqrt{m\eta},
 \qquad 1\le j\le m.
\end{equation}
\end{lemma}

\begin{proof}
Let $\mathbf A,\mathbf B:\C^m\to H$ be the operators with columns
$A_j,B_j$.  The entrywise hypothesis gives
\[
 \|\mathbf A^*\mathbf A-\mathbf B^*\mathbf B\|\le m\eta.
\]
For positive semidefinite operators $X,Y$,
\begin{equation}\label{eq:square-root-holder}
 \|X^{1/2}-Y^{1/2}\|\le\|X-Y\|^{1/2}.
\end{equation}
Indeed, put $Z=X^{1/2}-Y^{1/2}$ and choose a unit vector $u$ with
$Zu=\lambda u$, where, after interchanging $X$ and $Y$ if necessary,
$\lambda=\|Z\|\ge0$.  Since
\[
 X-Y=X^{1/2}Z+ZY^{1/2},
\]
we have
\[
 \langle(X-Y)u,u\rangle
 =\lambda\langle(X^{1/2}+Y^{1/2})u,u\rangle
 \ge\lambda^2,
\]
which proves \eqref{eq:square-root-holder}.

Take polar decompositions
\[
 \mathbf A=W_A|\mathbf A|,
 \qquad
 \mathbf B=W_B|\mathbf B|.
\]
Because $\dim H\ge m$, the partial isometries may be extended to
isometries $W_A,W_B:\C^m\to H$.  Choose a unitary $U$ on $H$ with
$UW_B=W_A$.  Then
\[
 \mathbf A-U\mathbf B
 =W_A(|\mathbf A|-|\mathbf B|),
\]
and \eqref{eq:square-root-holder} gives
\[
 \|\mathbf A-U\mathbf B\|
 \le\|\mathbf A^*\mathbf A-\mathbf B^*\mathbf B\|^{1/2}
 \le\sqrt{m\eta}.
\]
Applying this operator to the standard basis vectors proves
\eqref{eq:gram-matching}.
\end{proof}

Finally we prove separately the quantitative nondegeneracy of composition with a
nonconstant disc at the level of finite jets.

\begin{lemma}\label{lem:jet-composition}
Fix $m,N\ge1$ and $M>0$.  Let
$\phi:(\C,0)\to(\C^N,0)$ be a holomorphic germ whose Taylor
coefficients through order $m$ have modulus at most $M$, and put
$s=\|\phi'(0)\|>0$.  Let $\mathcal P_m$ be the space of holomorphic
polynomials on $\C^N$ of degree at most $m$, and define
\[
 J_\phi:\mathcal P_m\longrightarrow\C^{m+1},
 \qquad
 J_\phi(P)=\bigl([\zeta^k](P\circ\phi)\bigr)_{k=0}^m.
\]
Put
\[
 D=\frac{m(m+1)}2.
\]
Then $J_\phi$ is surjective.  If $A_\phi$ is an
$(m+1)\times(m+1)$ minor of maximal determinant modulus, then
\begin{equation}\label{eq:jet-minor}
 |\det A_\phi|\ge c s^D,
 \qquad
 \|A_\phi^{-1}\|\le C s^{-D},
\end{equation}
where $c,C>0$ depend only on $m,N,M$.  Moreover, for
\[
 B_\phi=A_\phi^{-1}J_\phi,
\]
every matrix entry of $B_\phi$ has modulus at most $1$, and the columns
corresponding to $A_\phi$ form the identity matrix.
\end{lemma}

\begin{proof}
Choose a coordinate $\ell$ such that
\[
 |\phi_\ell'(0)|\ge\frac{s}{\sqrt N}.
\]
On the subspace spanned by
\[
 1,v_\ell,v_\ell^2,\ldots,v_\ell^m,
\]
the matrix of $J_\phi$ is triangular: if
$\phi_\ell(\zeta)=a\zeta+O(\zeta^2)$, then
\[
 \phi_\ell(\zeta)^k=a^k\zeta^k+O(\zeta^{k+1}).
\]
Thus the corresponding determinant has modulus
\[
 |a|^{1+2+\cdots+m}
 \ge c s^{m(m+1)/2}=c s^D.
\]
This proves surjectivity and the first inequality in
\eqref{eq:jet-minor} for a maximal minor $A_\phi$.

The entries of the matrix of $J_\phi$ are bounded in terms of
$m,N,M$, so the cofactors of $A_\phi$ are uniformly bounded.  Cramer's
rule gives the second estimate in \eqref{eq:jet-minor}.  It also shows
that every entry of $B_\phi=A_\phi^{-1}J_\phi$ has modulus at most $1$:
replacing one selected column of $A_\phi$ by any other column of
$J_\phi$ produces the numerator determinant, whose modulus is at most
$|\det A_\phi|$ by maximality.  The selected columns themselves form
the identity matrix.
\end{proof}

\begin{proof}[Proof of Proposition~\ref{prop:finite-type-power}]
We may assume $p=0$ and put
\[
 T=\Delta_1(p)<\infty.
\]
Choose once and for all an integer
\begin{equation}\label{eq:choose-m}
 m>T^N.
\end{equation}
Since $T\ge1$, this also gives $2m>T$.  All constants below may depend
on this fixed $m$.

Let $f:\D\to\Omega$ be holomorphic, with center $z=f(0)$ sufficiently
near $p$, and write
\[
 \phi=f-z,
 \qquad
 g=-r\circ f>0,
 \qquad
 \delta=g(0)=-r(z).
\]
We shall prove
\begin{equation}\label{eq:disc-power-bound}
 \|f'(0)\|\le C\delta^\varepsilon
\end{equation}
for a fixed $\varepsilon>0$.

First we convert the mean estimate into finite-jet control.  Boundedness
of $\Omega$, Cauchy's inequalities, and smoothness of $r$ give a
uniform $C^{4m+1}$ bound for $g$ on the fixed disc $\D_a$, after
decreasing $a$ if necessary.  Lemma~\ref{lem:disc-mean} gives
\[
 \int_{\D_\rho}g\,dA\le C\delta\rho^2,
 \qquad 0<\rho\le a.
\]
Apply Lemma~\ref{lem:finite-jet-control} with $q=4m$ and $\ell=2m$.
Since
\[
 1-\frac{2m}{4m+1}=\frac{2m+1}{4m+1}>\frac12,
\]
we obtain
\begin{equation}\label{eq:jet-coeff}
 \bigl|[\zeta^j\bar\zeta^k]g\bigr|
 \le C\delta^{1/2},
 \qquad j+k\le2m.
\end{equation}

For $z$ in the closure of a sufficiently small fixed neighborhood of
$p$, let $R_z(v)$ be the Taylor polynomial of
\[
 r(z+v)-r(z)
\]
of total degree $2m$.  Lemma~\ref{lem:hermitian-decomposition} gives a
fixed Hermitian space $H$ and holomorphic polynomials, depending
continuously on $z$,
\begin{equation}\label{eq:jet-hermitian}
 R_z(v)=2\operatorname{Re} h_z(v)+\|F_z(v)\|^2-\|G_z(v)\|^2,
\end{equation}
all vanishing at $v=0$.  We use the construction in that lemma, so $H$
contains the monomials of degrees $1$ through $2m-1$; in particular
$\dim H\ge m$.

At the boundary point $p$, finite type gives
\[
 \operatorname{ord}_0(r\circ\gamma)
 \le T\operatorname{ord}_0\gamma
\]
for every nonconstant holomorphic curve germ through $p$.  Since
$2m>T$, Lemma~\ref{lem:unitary-colength}, applied to $R_p$, shows that
for every unitary $U$ on $H$ the ideal
\[
 I_U=(h_p,F_p-UG_p)\subset\mathcal O_0
\]
satisfies
\begin{equation}\label{eq:jet-length}
 \dim_\C\mathcal O_0/I_U\le T^N<m.
\end{equation}
This is the only use of finite type in the proof.

We now return to the disc.  Expand
\[
 F_z(\phi(\zeta))=\sum_{j\ge1}a_j\zeta^j,
 \qquad
 G_z(\phi(\zeta))=\sum_{j\ge1}b_j\zeta^j.
\]
Because $R_z$ is the $2m$-jet of $r(z+v)-r(z)$ and $\phi(0)=0$, the
coefficients through total degree $2m$ of $R_z\circ\phi$ agree with
those of $r\circ f-r(z)$.  Hence \eqref{eq:jet-coeff} gives
\begin{equation}\label{eq:jet-grams}
 |\langle a_j,a_k\rangle-\langle b_j,b_k\rangle|
 \le C\delta^{1/2},
 \qquad 1\le j,k\le m.
\end{equation}
The pure holomorphic coefficient of degree $j$ in
\eqref{eq:jet-hermitian} comes only from $h_z\circ\phi$, so also
\begin{equation}\label{eq:h-small}
 |[\zeta^j](h_z\circ\phi)|\le C\delta^{1/2},
 \qquad 1\le j\le m.
\end{equation}

Apply Lemma~\ref{lem:gram-matching} to the two lists
$a_1,\ldots,a_m$ and $b_1,\ldots,b_m$.  There is a unitary $U_f$ on
$H$ such that the Taylor coefficients through degree $m$ of
\begin{equation}\label{eq:jet-small-generators}
 h_z\circ\phi,
 \qquad
 (F_z-U_fG_z)\circ\phi
\end{equation}
are $O(\delta^{1/4})$.  Here the coefficients of $h_z\circ\phi$ are in
fact $O(\delta^{1/2})$, which is stronger.  Let $q_z$ denote $h_z$ or a
component of $F_z-U_fG_z$.  Since the coefficients of $\phi$ through
every fixed order are uniformly bounded, the identity
\[
 (v^\beta q_z)\circ\phi=\phi^\beta(q_z\circ\phi)
\]
shows that the Taylor coefficients through order $m$ of
$(v^\beta q_z)\circ\phi$ are also $O(\delta^{1/4})$ for every
$|\beta|\le m$.

Let $\mathcal P_m$ be the holomorphic polynomials in $v\in\C^N$ of
degree at most $m$, and define
\[
 J_f:\mathcal P_m\longrightarrow\C^{m+1},
 \qquad
 J_f(P)=\bigl([\zeta^k](P\circ\phi)\bigr)_{k=0}^m.
\]
Put
\[
 s=\|f'(0)\|,
 \qquad
 D=\frac{m(m+1)}2.
\]
If $s=0$, \eqref{eq:disc-power-bound} is trivial, so assume $s>0$.
The Taylor coefficients of $\phi$ through order $m$ are uniformly
bounded, and Lemma~\ref{lem:jet-composition} gives a maximal minor
$A_f$ and a normalized map
\[
 B_f=A_f^{-1}J_f
\]
such that
\begin{equation}\label{eq:jet-inverse}
 \|A_f^{-1}\|\le Cs^{-D},
\end{equation}
every entry of $B_f$ has modulus at most $1$, and the selected columns
of $B_f$ form the identity matrix.

Set
\begin{equation}\label{eq:epsilon-choice}
 \varepsilon=\frac1{2m(m+1)},
\end{equation}
so that $\varepsilon D=1/4$.  Suppose, toward a contradiction, that
\eqref{eq:disc-power-bound} does not hold uniformly for discs centered
near $p$.  Then there are discs $f_j$ with centers $z_j\to p$ such that,
with
\[
 \delta_j=-r(z_j),
 \qquad
 s_j=\|f_j'(0)\|>0,
\]
we have
\begin{equation}\label{eq:jet-counter}
 \frac{s_j}{\delta_j^\varepsilon}\longrightarrow\infty.
\end{equation}
For each $j$, let $U_j=U_{f_j}$.  After passing to a subsequence,
compactness of the unitary group gives
\[
 U_j\longrightarrow U.
\]
There are only finitely many choices of the selected columns defining
the maximal minor $A_{f_j}$, so we may also assume that the same columns
are selected for every $j$.  Since all entries of $B_{f_j}$ have
modulus at most $1$, after a further subsequence
\[
 B_{f_j}\longrightarrow B.
\]
The selected columns of $B$ still form the identity; hence
\begin{equation}\label{eq:B-rank}
 \operatorname{rank}B=m+1.
\end{equation}

Let $q_j$ be one of the generators
\[
 h_{z_j}
 \quad\text{or a component of}\quad
 F_{z_j}-U_jG_{z_j}.
\]
Put $\phi_j=f_j-z_j$.  For $|\beta|\le m$, let $[v^\beta q_j]_{\le m}$ denote its Taylor
truncation in the $v$ variables at degree $m$.  Since
$\phi_j(0)=0$, terms of $v$-degree at least $m+1$ contribute only terms
of $\zeta$-order at least $m+1$ after composition with $\phi_j$.
Therefore truncation does not change the vector represented by
$J_{f_j}$.  The estimate following \eqref{eq:jet-small-generators}
gives
\[
 \|J_{f_j}([v^\beta q_j]_{\le m})\|
 \le C\delta_j^{1/4}.
\]
Using \eqref{eq:jet-inverse},
\[
 \|B_{f_j}([v^\beta q_j]_{\le m})\|
 \le Cs_j^{-D}\delta_j^{1/4}
 =C\left(\frac{\delta_j^\varepsilon}{s_j}\right)^D
 \longrightarrow0
\]
by \eqref{eq:jet-counter}.

The coefficients of $h_{z_j}$ converge to those of $h_p$, while
$F_{z_j}-U_jG_{z_j}$ converges coefficientwise to $F_p-UG_p$.
Thus $q_j$ converges to the corresponding generator $q$ of $I_U$, and
$B_{f_j}\to B$ gives
\begin{equation}\label{eq:B-kills-products}
 B([v^\beta q]_{\le m})=0
\end{equation}
for every generator $q$ of $I_U$ and every $|\beta|\le m$.

Let $\mathfrak m$ be the maximal ideal of $\mathcal O_0$ and identify
\[
 \mathcal P_m\simeq\mathcal O_0/\mathfrak m^{m+1}.
\]
The classes of the products $v^\beta q$ appearing in
\eqref{eq:B-kills-products} span the image of $I_U$ in this quotient.
Indeed, if
\[
 h=\sum_\nu a_\nu q_\nu\in I_U,
\]
replace each $a_\nu$ by its Taylor polynomial of degree $m$.  The
discarded part lies in $\mathfrak m^{m+1}$, and multiplication by
$q_\nu\in\mathfrak m$ leaves it in $\mathfrak m^{m+1}$; the retained
part is a linear combination of monomials $v^\beta q_\nu$ with
$|\beta|\le m$.  Hence \eqref{eq:B-kills-products} says that $B$
vanishes on the entire image of $I_U$ in
$\mathcal O_0/\mathfrak m^{m+1}$.

By \eqref{eq:B-rank}, $B$ is onto $\C^{m+1}$, so it induces a
surjective map
\[
 \frac{\mathcal O_0}{I_U+\mathfrak m^{m+1}}
 \longrightarrow\C^{m+1}.
\]
Consequently
\[
 m+1
 \le\dim_\C\frac{\mathcal O_0}{I_U+\mathfrak m^{m+1}}
 \le\dim_\C\frac{\mathcal O_0}{I_U}
 \le T^N<m,
\]
where the last inequality is \eqref{eq:jet-length}.  This contradiction
proves \eqref{eq:disc-power-bound} with the exponent
\eqref{eq:epsilon-choice}.

Finally, $-r(z)\asymp\delta_\Omega(z)$ near the boundary.  If an
admissible disc in the definition of the Kobayashi--Royden metric
satisfies $f(0)=z$ and $f'(0)=a^{-1}X$, then
\[
 a^{-1}\|X\|\le C\delta_\Omega(z)^\varepsilon,
\]
and therefore
\[
 a\ge c\frac{\|X\|}{\delta_\Omega(z)^\varepsilon}.
\]
Taking the infimum over all admissible discs proves
\eqref{eq:weight-power}.
\end{proof}

\begin{proof}[Proof of Theorem~\ref{thm:finite-type}]
For every $p\in\partial\Omega$, Proposition~\ref{prop:finite-type-power} gives a neighborhood $V_p$ and constants $c_p,\varepsilon_p>0$ such that
\[
 \kappa_\Omega(z;X)\ge c_p\frac{\|X\|}{\delta_\Omega(z)^{\varepsilon_p}},
 \qquad z\in V_p\cap\Omega.
\]
Choose a finite subcover of the compact boundary.  Restricting to $\delta_\Omega(z)<1$, let $\varepsilon$ be the minimum of the finitely many exponents and $c$ the minimum of the corresponding constants.  Since $0<\delta_\Omega(z)<1$ and $\varepsilon\le\varepsilon_p$, each local estimate implies
\[
 \kappa_\Omega(z;X)\ge c\frac{\|X\|}{\delta_\Omega(z)^\varepsilon}.
\]
This gives the desired global power estimate near the boundary.  On the remaining compact subset of the bounded domain, the Kobayashi--Royden metric dominates a positive multiple of the Euclidean norm, so decreasing $c$ extends the estimate to all of $\Omega$.

Since the boundary is $C^\infty$, choose an integer $q\ge3$ with $q\varepsilon\ge1$.  Proposition~\ref{prop:self-general} then yields the normal estimate
\[
 \kappa_\Omega(z;X)
 \ge
 c\,
 \frac{|\partial r(z)X|}
 {(-r(z))\log(A/(-r(z)))}
\]
near $\partial\Omega$.  Proposition~\ref{prop:criterion} proves Kobayashi completeness.
\end{proof}

\section{Subharmonic functions and polynomial discs}
Let $\omega$ be an even, nonnegative smooth function supported in
$(-1,1)$ with $\int_\R\omega(t)\,dt=1$.  Put
\[
 \chi(x)=\int_\R\max\{x-t,0\}\omega(t)\,dt.
\]
Then $\chi\in C^\infty(\R)$ is convex and nondecreasing,
$0\le\chi'\le1$, $\chi(x)=0$ for $x\le-1$, and
$\chi(x)=x$ for $x\ge1$.  Since $\omega$ is even,
$0\le\chi(x)-\max\{x,0\}\le\chi(0)<1/2$.
The only features of $\chi$ that will be used are its convexity,
monotonicity, and the fact that it agrees exactly with the two affine
pieces $0$ and $x$ away from the transition interval $[-1,1]$.

\subsection{A two-pole subharmonic function}
For $a\ge1$ put
\[
\tau_a=\frac1{2a},\qquad B_a=\frac9{16a},\qquad
L_a=\log\frac1{B_a(B_a-\tau_a)}=\log\frac{256a^2}{9}.
\]
The two distinguished parameter values are $0$ and $\tau_a$.  The
quantity $B_a$ is chosen slightly to the right of $\tau_a$:
\[
 B_a-\tau_a=\frac1{16a}.
\]
This makes the vertical line $\{\Ree\zeta=B_a\}$ a convenient place at
which the two definitions below will agree exactly.

Away from $0$ and $\tau_a$, define
\[
g_a(\zeta)=7-8a\Ree\zeta+
 \frac{\log|\zeta(\zeta-\tau_a)|}{L_a},
\]
and let
\begin{equation}\label{eq:hAdef}
h_a(\zeta)=\frac18
\begin{cases}
\chi(g_a(\zeta)),&\Ree\zeta\le B_a,\\
g_a(\zeta),&\Ree\zeta\ge B_a,
\end{cases}
\end{equation}
with value $0$ at both logarithmic poles.

The affine term $7-8a\Ree\zeta$ is the part which will later produce
$1-a\Ree\zeta$ after division by $8$.  The logarithmic term has two
poles, at precisely the two parameter values which are mapped to the
node.  Composing with $\chi$ removes these poles smoothly because
$\chi$ is identically zero when its argument is at most $-1$.

\begin{lemma}\label{lem:scalar}
For all sufficiently large $a$, the function $h_a$ is smooth and
subharmonic on $\C$. With $r_a=(2a)^{-20}$,
\begin{equation}\label{eq:scalarzero}
h_a=0\quad\text{on }\{|\zeta|\le r_a\}\cup
\{|\zeta-\tau_a|\le r_a\},
\end{equation}
and
\begin{equation}\label{eq:scalarineq}
h_a(\zeta)\le1-a\Ree\zeta-\frac1{16}
\qquad(|\zeta|\le1).
\end{equation}
There are constants $C,C_m$ independent of $a$ such that
\begin{equation}\label{eq:scalarbounds}
|h_a(\zeta)|\le Ca(1+|\zeta|),\qquad
\sup_\C|D^mh_a|\le C_ma^{20m+1}\quad(m\ge1).
\end{equation}
\end{lemma}

\begin{proof}
There are four points to check: the two formulas in
\eqref{eq:hAdef} must glue smoothly, the logarithmic poles must be
removed, the affine upper bound must hold, and the derivatives must be
controlled uniformly in $a$.

\smallskip
\noindent\emph{Step 1: the two formulas agree near the joining line.}
On the line $\Ree\zeta=B_a$ we have
\[
 |\zeta(\zeta-\tau_a)|\ge B_a(B_a-\tau_a).
\]
By the definition of $L_a$,
\[
 \frac{\log|\zeta(\zeta-\tau_a)|}{L_a}\ge -1.
\]
Moreover,
\[
 8aB_a=\frac92.
\]
Therefore
\[
 g_a(\zeta)\ge 7-\frac92-1=\frac32>1.
\]
Since $\chi(x)=x$ for $x\ge1$, we have
$\chi(g_a)=g_a$ in a neighborhood of the joining line.  Hence the two
pieces in \eqref{eq:hAdef} agree there to all orders.  Away from the two
poles, $g_a$ is harmonic; because $\chi$ is convex and nondecreasing,
$\chi\circ g_a$ is subharmonic.  Thus only the poles remain to be
checked.

\smallskip
\noindent\emph{Step 2: the logarithmic poles are killed by the cutoff.}
Assume
\[
 \min\{|\zeta|,|\zeta-\tau_a|\}\le(2a)^{-20}.
\]
The other distance is at most $\tau_a+(2a)^{-20}=(2a)^{-1}+(2a)^{-20}$,
so, for large $a$,
\[
 |\zeta(\zeta-\tau_a)|\le 2(2a)^{-21}.
\]
Also $\Ree\zeta\ge-(2a)^{-20}$.  Hence
\[
 g_a(\zeta)
 \le 7+4(2a)^{-19}
 +\frac{\log2-21\log(2a)}{L_a}.
\]
Since
\[
 L_a=2\log a+O(1),
\]
the last expression tends to
\[
 7-\frac{21}{2}=-\frac72<-1.
\]
Thus $g_a<-1$ throughout both sufficiently small discs for large $a$.
They also lie in the half-plane $\Ree\zeta<B_a$.  Consequently
$\chi(g_a)=0$ there and therefore $h_a=0$.  This proves
\eqref{eq:scalarzero}; in particular the apparent logarithmic
singularities are removable and $h_a$ is globally smooth and
subharmonic.

The exponent $20$ has no intrinsic meaning.  What matters is that the
small zero-neighborhoods shrink by a fixed sufficiently large power of
$a$.  We keep the value $20$ because it leaves a large numerical margin
in the later separation estimates.

\smallskip
\noindent\emph{Step 3: the affine upper bound.}
For $|\zeta|\le1$,
\[
 |\zeta(\zeta-\tau_a)|\le 1+\tau_a,
\]
and therefore, for all sufficiently large $a$,
\[
 \frac{\log|\zeta(\zeta-\tau_a)|}{L_a}
 \le\frac{\log(1+\tau_a)}{L_a}\le\frac12.
\]
Hence
\[
 g_a(\zeta)\le \frac{15}{2}-8a\Ree\zeta.
\]
If $\Ree\zeta\le B_a$, then
\[
 \frac{15}{2}-8a\Ree\zeta
 \ge\frac{15}{2}-8aB_a
 =\frac{15}{2}-\frac92=3>1.
\]
By monotonicity of $\chi$ and the identity $\chi(x)=x$ for $x\ge1$,
\[
 \chi(g_a(\zeta))
 \le \chi\!\left(\frac{15}{2}-8a\Ree\zeta\right)
 =\frac{15}{2}-8a\Ree\zeta.
\]
After division by $8$,
\[
 h_a(\zeta)\le \frac{15}{16}-a\Ree\zeta
 =1-a\Ree\zeta-\frac1{16}.
\]
If $\Ree\zeta\ge B_a$, then $h_a=g_a/8$, so the same estimate follows
directly.  This proves \eqref{eq:scalarineq}.

\smallskip
\noindent\emph{Step 4: growth and derivative estimates.}
Outside the zero-neighborhoods in \eqref{eq:scalarzero}, both
$|\zeta|$ and $|\zeta-\tau_a|$ are at least $(2a)^{-20}$.  Thus
\[
 \bigl|\log|\zeta(\zeta-\tau_a)|\bigr|
 \le C\log a+2\log(1+|\zeta|).
\]
Since $L_a\asymp\log a$, the logarithmic contribution to $g_a$ is
$O(1+\log(1+|\zeta|))$, while the affine contribution is
$O(a(1+|\zeta|))$.  This gives the first estimate in
\eqref{eq:scalarbounds}.

For a derivative of positive order $m$, differentiating
$\log|\zeta|$ or $\log|\zeta-\tau_a|$ produces a factor bounded by a
constant times the inverse $m$th power of the corresponding distance.
Hence, outside the zero-neighborhoods,
\[
 |D^m\log|\zeta(\zeta-\tau_a)||\le C_m a^{20m}.
\]
The derivatives of $\chi$ are bounded independently of $a$.  Repeated
application of the chain rule therefore gives
\[
 |D^mh_a|\le C_ma^{20m+1}.
\]
On the two zero-neighborhoods all derivatives vanish.  This proves
\eqref{eq:scalarbounds} and completes the lemma.
\end{proof}

\subsection{Analytic discs}
We now place the scalar function $h_a$ on a family of polynomial discs.
The polynomial curve
\[
 \Gamma(\xi)=\bigl(\xi^2(\xi-1),\xi(\xi-1)\bigr)
 =f(\xi)(\xi,1),\qquad f(\xi)=\xi(\xi-1),
\]
has a node at the origin:
\[
 \Gamma(0)=\Gamma(1)=0.
\]
After rescaling the parameter, these two preimages will be separated by
$n^{-2}$.  The series $\sum n^{-2}$ converges, which is exactly what is
needed later for a finite Kobayashi length.  A second scale $b_n$ shrinks
the tangential image rapidly enough that different discs can be
separated by a cutoff while the eventual infinite series still converges
smoothly.

For all sufficiently large integers $n$, put
\begin{equation}\label{eq:parameters}
\begin{aligned}
 a_n&=\frac{n^2}{2},& b_n&=e^{-n^{3/4}},& \tau_n&=\frac1{2a_n},\\
 \rho_n&=h_{a_n},& f(\xi)&=\xi(\xi-1).
\end{aligned}
\end{equation}
Thus
\[
 2a_n=n^2,\qquad \tau_n=n^{-2}.
\]
Consider the polynomial discs
\begin{equation}\label{eq:gamma}
 \gamma_n(\zeta)=b_n\Gamma(2a_n\zeta)
 =b_nf(2a_n\zeta)(2a_n\zeta,1).
\end{equation}
At $\zeta=0$ and $\zeta=\tau_n$ the argument of $\Gamma$ is respectively
$0$ and $1$, so
\[
 \gamma_n(0)=\gamma_n(\tau_n)=0.
\]
Set $K_n=\gamma_n(\overline\D)$ and $R_n=3b_n n^6$.  Since
$|2a_n\zeta|\le n^2$ on $\overline\D$, the two components of
$\Gamma(2a_n\zeta)$ have size $O(n^6)$, and therefore
\[
 K_n\subset B(0,R_n).
\]
Because $b_n=e^{-n^{3/4}}$ decays faster than any negative power of
$n$, $R_n\to0$ and is eventually decreasing.

The next rational function is designed so that it is constant on each
curve $K_\ell$, with a different constant for each $\ell$.  Put
\begin{equation}\label{eq:pi}
 Q(s,t)=s(s-t),\qquad
 \pi_n(s,t)=\frac{b_nQ(s,t)}{t^3}\quad(t\ne0).
\end{equation}
To see why this is useful, write on the $\ell$th disc
\[
 \xi=2a_\ell\zeta,
 \qquad
 (s,t)=b_\ell f(\xi)(\xi,1).
\]
Then
\[
 s=b_\ell\xi f(\xi),\qquad t=b_\ell f(\xi),
\]
so
\[
 Q(s,t)=s(s-t)
 =b_\ell^2\xi(\xi-1)f(\xi)^2
 =b_\ell^2 f(\xi)^3.
\]
Since $t^3=b_\ell^3f(\xi)^3$, we obtain
\begin{equation}\label{eq:pi-discs}
 \pi_n\circ\gamma_\ell=\frac{b_n}{b_\ell},\qquad
 \frac{s}{2a_n t}\circ\gamma_n(\zeta)=\zeta.
\end{equation}
The first identity separates the different curves; the second recovers
the original parameter on the $n$th curve.

Fix a smooth radial function $\vartheta:\C\to[0,1]$, equal to $1$
on $|\zeta|\le1/2$ and to $0$ on $|\zeta|\ge1$.  Define
\begin{equation}\label{eq:vk}
 v_n(s,t)=\rho_n\!\left(\frac{s}{2a_n t}\right)
 \vartheta\bigl(10n^{1/4}(\pi_n(s,t)-1)\bigr)\quad(t\ne0),
\end{equation}
and set $v_n=0$ on $t=0$.  Finally put
\begin{equation}\label{eq:epsrho}
 \varepsilon_n=\frac12b_n (2a_n)^{-38}.
\end{equation}
The exponent $38$ is twice $19$.  This is exactly what is needed below:
if the product $|\xi||\xi-1|$ is at most $(2a_n)^{-38}$, then one of the
two factors is at most $(2a_n)^{-19}$, which places
$\xi/(2a_n)$ inside one of the zero-neighborhoods of $\rho_n$ from
Lemma~\ref{lem:scalar}.

\begin{lemma}\label{lem:local}
For all sufficiently large $n$, the function $v_n$ is smooth on $\C^2$ and
\begin{equation}\label{eq:vzero}
 v_n=0\quad\text{on }\{|t|\le\varepsilon_n\},\qquad
 v_n\circ\gamma_\ell=
 \begin{cases}
 \rho_n,&\ell=n,\\
 0,&\ell\ne n,
 \end{cases}
\end{equation}
for every sufficiently large $\ell$.
\end{lemma}

\begin{proof}
There are two independent assertions: vanishing near the hyperplane
$t=0$, which gives smoothness there, and separation of the different
curves.

\smallskip
\noindent\emph{Vanishing near $t=0$.}
Suppose first that the cutoff in \eqref{eq:vk} is nonzero.  Then
\[
 |10n^{1/4}(\pi_n-1)|<1,
\]
so in particular $|\pi_n|\le2$ for large $n$.  If
$|t|\le\varepsilon_n$ and we write $\xi=s/t$, then
\[
 |f(\xi)|=|\xi(\xi-1)|
 =\frac{|Q(s,t)|}{|t|^2}
 =\frac{|t\pi_n(s,t)|}{b_n}
 \le \frac{2\varepsilon_n}{b_n}
 =(2a_n)^{-38}.
\]
Therefore
\[
 |\xi|\,|\xi-1|\le(2a_n)^{-38}.
\]
At least one of the two factors is at most $(2a_n)^{-19}$.  If
$|\xi|\le(2a_n)^{-19}$, then
\[
 \left|\frac{\xi}{2a_n}\right|\le(2a_n)^{-20};
\]
if $|\xi-1|\le(2a_n)^{-19}$, then
\[
 \left|\frac{\xi}{2a_n}-\frac1{2a_n}\right|
 \le(2a_n)^{-20}.
\]
In either case Lemma~\ref{lem:scalar} gives
\[
 \rho_n\!\left(\frac{\xi}{2a_n}\right)=0.
\]
Hence $v_n=0$.  If the cutoff itself is zero, then of course $v_n=0$ as
well.  Thus $v_n$ vanishes on the whole slab
$\{|t|\le\varepsilon_n\}$, proving both the first assertion in
\eqref{eq:vzero} and smoothness across $t=0$.

\smallskip
\noindent\emph{Separation of the curves.}
Because $b_j=e^{-j^{3/4}}$, neighboring ratios satisfy
\[
 \frac{b_n}{b_{n-1}}
 =\exp\bigl(-(n^{3/4}-(n-1)^{3/4})\bigr),
\]
and
\[
 \frac{b_n}{b_{n+1}}
 =\exp\bigl((n+1)^{3/4}-n^{3/4}\bigr).
\]
By the mean value theorem,
\[
 n^{3/4}-(n-1)^{3/4}
 =\frac34\xi_n^{-1/4}
\]
for some $\xi_n\in(n-1,n)$, and similarly for the second difference.
Thus both differences are comparable to $n^{-1/4}$.  For all sufficiently large $n$, both increments are bounded below by
$\frac23n^{-1/4}$.  Using $1-e^{-x}\ge x/2$ and $e^x-1\ge x$
for positive $x$ sufficiently small, we therefore obtain
\begin{equation}\label{eq:neighbor-ratios}
 1-\frac{b_n}{b_{n-1}}\ge\frac1{3n^{1/4}},\qquad
 \frac{b_n}{b_{n+1}}-1\ge\frac1{3n^{1/4}},\qquad
 1-\frac{b_{n+1}}{b_n}\ge\frac1{3n^{1/4}}.
\end{equation}
Monotonicity of $b_j$ then gives
\[
 \left|\frac{b_n}{b_\ell}-1\right|
 \ge\frac1{3n^{1/4}}
 \qquad(\ell\ne n).
\]
On $K_\ell\setminus\{0\}$, identity \eqref{eq:pi-discs} shows that the
argument of the cutoff is
\[
 10n^{1/4}\left(\frac{b_n}{b_\ell}-1\right).
\]
Its modulus is at least $10/3>1$, so the cutoff vanishes.  Hence
$v_n|_{K_\ell}=0$ for $\ell\ne n$.

On $K_n\setminus\{0\}$ we have $\pi_n=1$, so the cutoff is identically
one, while the second identity in \eqref{eq:pi-discs} gives
$s/(2a_nt)=\zeta$.  Consequently
\[
 v_n\circ\gamma_n=\rho_n
\]
away from the two parameters mapping to the origin.  At those two
parameters both sides vanish by Lemma~\ref{lem:scalar}.  This proves
\eqref{eq:vzero}.
\end{proof}

\section{Construction of the plurisubharmonic function}
The functions $v_n$ already have exactly the restrictions we want on the
polynomial curves.  Their defect is transverse: the cutoff involving
$\pi_n$ need not preserve plurisubharmonicity.  The correction will be
constructed in three stages.

First, we build a polynomial $T_n$ which vanishes on a whole block of
curves around $K_n$ but is uniformly large wherever the cutoff in $v_n$
can create a negative Levi direction.  Second, $\log|T_n|$ is combined
with a strictly plurisubharmonic radial term and smoothed by $\chi$; this
produces a nonnegative plurisubharmonic function $q_n$ which is zero near
the relevant curves and has a uniformly positive Levi form precisely
where it is needed.  Finally, the corrected local function is extended
to all of $\C^2$ by a smooth maximum with a logarithm.

The numerical parameters below are deliberately separated.  Their role
is not to optimize constants, but to make three different requirements
hold simultaneously: the correction must vanish on enough neighboring
curves, it must vanish on a ball containing all sufficiently later
curves, and all derivatives must grow subexponentially relative to the
final coefficient $2^{-n}$.

\subsection{Polynomials defining the curves}
Choose
\begin{equation}\label{eq:index-sets}
\begin{aligned}
 M_n&=\lceil400n^{1/4}\log n\rceil,\qquad N_n=2M_n,\\
 I_n&=\{n-M_n,\ldots,n+N_n\},\qquad m_n=|I_n|.
\end{aligned}
\end{equation}
Thus $m_n=3M_n+1=O(n^{1/4}\log n)$.  Set
\[
 S_n=R_{n-M_n},\qquad U_n=B(0,4S_n),
\]
and
\begin{equation}\label{eq:rk}
 r_n=b_{n+M_n}.
\end{equation}
The radius $S_n$ is large enough to contain every curve with index in
$I_n$, while $r_n$ will be the radius of a ball on which the correction
vanishes.  We now give the scale comparisons which justify these two
claims and which will also control derivatives after rescaling.

Define
\[
\begin{aligned}
 \Delta_n^-&=n^{3/4}-(n-M_n)^{3/4},\\
 \Delta_n^+&=(n+M_n)^{3/4}-n^{3/4},\\
 \Delta_n^*&=(n+N_n+1)^{3/4}-n^{3/4}.
\end{aligned}
\]
Since $M_n=o(n)$, Taylor expansion of $x^{3/4}$ at $n$ gives
\[
 (n+h)^{3/4}
 =n^{3/4}+\frac34n^{-1/4}h
 +O(n^{-5/4}h^2)
\]
uniformly for $|h|\le2M_n+1$.  Because
$M_n=400n^{1/4}\log n+O(1)$,
\[
 n^{-5/4}M_n^2=O(n^{-3/4}(\log n)^2)=o(1).
\]
Therefore
\begin{align*}
 \Delta_n^-&=300\log n+o(1),\\
 \Delta_n^+&=300\log n+o(1),\\
 \Delta_n^*&=600\log n+o(1).
\end{align*}
These three numbers measure, on the logarithmic scale, how much the
factors $b_j=e^{-j^{3/4}}$ change across the index block $I_n$.

Now
\[
 S_n=3b_{n-M_n}(n-M_n)^6,
\]
so
\[
 \frac{S_n}{b_n}
 =3(n-M_n)^6
 \exp\bigl(n^{3/4}-(n-M_n)^{3/4}\bigr)
 =3n^{306+o(1)}.
\]
Similarly,
\[
 \frac{r_n}{b_n}
 =\frac{b_{n+M_n}}{b_n}
 =e^{-\Delta_n^+}
 =n^{-300+o(1)}.
\]
Since $2a_n=n^2$, equation \eqref{eq:epsrho} reads
\[
 \frac{\varepsilon_n}{b_n}=\frac12n^{-76}.
\]
Thus
\[
 \frac{r_n}{\varepsilon_n}
 =2n^{-224+o(1)}\longrightarrow0,
\]
and in particular $r_n<\varepsilon_n$ for large $n$.
Finally,
\[
 \frac{R_{n+N_n+1}}{r_n}
 =3(n+N_n+1)^6
 \exp\bigl(-(n+N_n+1)^{3/4}+(n+M_n)^{3/4}\bigr).
\]
The exponent is
\[
 -(\Delta_n^*-\Delta_n^+)
 =-300\log n+o(1),
\]
so
\[
 \frac{R_{n+N_n+1}}{r_n}=3n^{-294+o(1)}\longrightarrow0.
\]
We have therefore proved the scale comparisons
\begin{equation}\label{eq:scale-comparisons}
 \frac{S_n}{b_n}\le Cn^6e^{\Delta_n^-},\qquad
 r_n<\varepsilon_n,\qquad
 R_{n+N_n+1}<r_n
\end{equation}
for all large $n$.

The meaning of these inequalities is worth keeping in mind.  Since the
radii $R_j$ are eventually decreasing, every $K_\ell$ with
$\ell\in I_n$ lies in $B(0,S_n)$.  On the other hand every sufficiently
later curve, namely every $K_\ell$ with $\ell>n+N_n$, lies inside
$B(0,r_n)$.  Finally $S_n/b_n$ is only polynomially large, which means
that the rescaled domain $b_n^{-1}U_n$ has polynomial radius.  This last
fact is what keeps the cutoff derivatives under control.

For $\ell\in I_n$ let
\[
 F_\ell(s,t)=t^3-b_\ell Q(s,t).
\]
On $K_\ell$, identity \eqref{eq:pi-discs} with $n=\ell$ says
$b_\ell Q/t^3=1$, hence $F_\ell=0$.  Define
\begin{equation}\label{eq:Tk}
 T_n(s,t)=e^4\left(\frac{20n^{1/4}}{\varepsilon_n^3}\right)^{m_n}
 \prod_{\ell\in I_n}F_\ell(s,t).
\end{equation}
Thus $T_n$ vanishes on every $K_\ell$ with $\ell\in I_n$.
The large prefactor is chosen so that the opposite happens on the
transition region of the cutoff in $v_n$: there $T_n$ will have modulus
at least $e^4$.

\begin{lemma}\label{lem:Tlower}
At every point where $v_n$ has a negative Levi direction,
$|T_n|\ge e^4$.
\end{lemma}

\begin{proof}
The only source of possible non-plurisubharmonicity in $v_n$ is the
cutoff
\[
 \vartheta\bigl(10n^{1/4}(\pi_n-1)\bigr).
\]
If this cutoff is locally constant, then $v_n$ is either zero or a
nonnegative constant multiple of
$\rho_n(s/(2a_nt))$.  The latter is plurisubharmonic because $\rho_n$ is
subharmonic and $s/(2a_nt)$ is holomorphic where $t\ne0$.  Hence a
negative Levi direction can occur only in the transition annulus of the
cutoff.  By Lemma~\ref{lem:local} it can also occur only where
$|t|\ge\varepsilon_n$.  Thus at such a point
\begin{equation}\label{eq:annulus}
 \frac1{20n^{1/4}}
 \le|\pi_n-1|
 \le\frac1{10n^{1/4}}.
\end{equation}

We claim that for every $\ell\in I_n$,
\[
 \left|1-\frac{b_\ell}{b_n}\pi_n\right|
 \ge\frac1{20n^{1/4}}.
\]
For $\ell=n$ this is exactly the first inequality in
\eqref{eq:annulus}.  Suppose $\ell<n$.  Then $b_\ell/b_n\ge1$, and
\eqref{eq:neighbor-ratios}, together with monotonicity, gives
\[
 1-\frac{b_n}{b_\ell}\ge\frac1{3n^{1/4}}.
\]
Hence
\begin{align*}
 \left|1-\frac{b_\ell}{b_n}\pi_n\right|
 &=\frac{b_\ell}{b_n}
   \left|\frac{b_n}{b_\ell}-\pi_n\right|\\
 &\ge\frac{b_\ell}{b_n}
 \left(1-\frac{b_n}{b_\ell}-|\pi_n-1|\right)\\
 &\ge \frac1{3n^{1/4}}-\frac1{10n^{1/4}}
 \ge\frac1{20n^{1/4}}.
\end{align*}
If $\ell>n$, then $b_\ell/b_n\le b_{n+1}/b_n$.  The third estimate in
\eqref{eq:neighbor-ratios} therefore gives the separate bound
\[
 1-\frac{b_\ell}{b_n}
 \ge1-\frac{b_{n+1}}{b_n}
 \ge\frac1{3n^{1/4}}.
\]
Consequently
\[
 \left|1-\frac{b_\ell}{b_n}\pi_n\right|
 \ge 1-\frac{b_\ell}{b_n}
       -\frac{b_\ell}{b_n}|\pi_n-1|
 \ge\frac1{20n^{1/4}}.
\]

Since
\[
 F_\ell=t^3\left(1-\frac{b_\ell}{b_n}\pi_n\right),
\]
every factor satisfies
\[
 |F_\ell|\ge\frac{|t|^3}{20n^{1/4}}.
\]
Substituting this into \eqref{eq:Tk},
\[
 |T_n|
 \ge e^4
 \left(\frac{20n^{1/4}}{\varepsilon_n^3}\right)^{m_n}
 \left(\frac{|t|^3}{20n^{1/4}}\right)^{m_n}
 =e^4\left(\frac{|t|}{\varepsilon_n}\right)^{3m_n}.
\]
Because $|t|\ge\varepsilon_n$, the last expression is at least $e^4$.
\end{proof}

\subsection{Construction of a plurisubharmonic function near the curves}
We next need two opposite estimates for the same polynomial $T_n$:
it must be very small on $B(0,r_n)$, but its derivatives on the much
larger ball $U_n$ must not grow too quickly.  Both are easiest after
rescaling by the natural size $b_n$ of the $n$th curve.

Write
\[
 (s,t)=b_n(x_1,x_2).
\]
Since $Q(b_nx_1,b_nx_2)=b_n^2x_1(x_1-x_2)$, formula
\eqref{eq:Tk} becomes
\begin{equation}\label{eq:Tscaled}
 T_n(b_nx)=e^4\left(20n^{1/4}
 \left(\frac{b_n}{\varepsilon_n}\right)^3\right)^{m_n}
 \prod_{\ell\in I_n}
 \left(x_2^3-\frac{b_\ell}{b_n}x_1(x_1-x_2)\right).
\end{equation}
Each factor has terms of degree $2$ and $3$.  Consequently every
monomial of the product has degree between $2m_n$ and $3m_n$.

For a polynomial $P(x)=\sum_\beta c_\beta x^\beta$ write
\[
 \|P\|_{\mathrm{coef}}=\sum_\beta|c_\beta|.
\]
This norm is submultiplicative.  For $\ell\in I_n$ we have
$\ell\ge n-M_n$, and hence
\[
 \frac{b_\ell}{b_n}
 =\exp(n^{3/4}-\ell^{3/4})
 \le e^{\Delta_n^-}.
\]
The coefficient norm of one factor in \eqref{eq:Tscaled} is therefore
at most $1+2e^{\Delta_n^-}\le Ce^{\Delta_n^-}$.  Also
\[
 \frac{b_n}{\varepsilon_n}=2(2a_n)^{38}=2n^{76},
\]
so its logarithm is $O(\log n)$.  Taking logarithms in
\eqref{eq:Tscaled} gives
\begin{equation}\label{eq:Tcoef}
 \log\|T_n(b_n\cdot)\|_{\mathrm{coef}}
 \le C m_n(\log n+\Delta_n^-)
 =O\bigl(n^{1/4}(\log n)^2\bigr).
\end{equation}

We now evaluate $T_n$ on the smaller ball $B(0,r_n)$.  In rescaled
coordinates this is
\[
 \|x\|\le\frac{r_n}{b_n}=e^{-\Delta_n^+}.
\]
On this ball
\[
 |x_2|^3\le e^{-3\Delta_n^+},
\]
while
\[
 \frac{b_\ell}{b_n}|x_1(x_1-x_2)|
 \le C e^{\Delta_n^-}e^{-2\Delta_n^+}.
\]
The latter bound dominates for large $n$, so every factor in
\eqref{eq:Tscaled} has modulus at most
\[
 C e^{\Delta_n^- -2\Delta_n^+}.
\]
Consequently
\[
 |T_n(b_nx)|\le e^4\left[
 C n^{1/4}\left(\frac{b_n}{\varepsilon_n}\right)^3
 e^{\Delta_n^- -2\Delta_n^+}\right]^{m_n}.
\]
Here
\[
 n^{1/4}\left(\frac{b_n}{\varepsilon_n}\right)^3
 =8n^{1/4+228}=8n^{228\frac14},
\]
and
\[
 e^{\Delta_n^- -2\Delta_n^+}
 =n^{-300+o(1)}.
\]
Thus the expression in square brackets is
\[
 n^{-71\frac34+o(1)},
\]
which tends to zero rapidly.  Since $m_n\to\infty$, this gives
\begin{equation}\label{eq:Tsmall}
 |T_n|\le e^{-2}\qquad\text{on }B(0,r_n)
\end{equation}
for all sufficiently large $n$.

For derivatives we use a standard elementary estimate.  If $P$ is a
polynomial of degree at most $d$ on a ball of radius $R\ge1$, then for
every fixed $m$
\[
 \|P\|_{C^m(B(0,R))}
 \le C_m(1+d)^m\|P\|_{\mathrm{coef}}(1+R)^d.
\]
In our case
\[
 b_n^{-1}U_n=B\!\left(0,\frac{4S_n}{b_n}\right),
 \qquad \frac{4S_n}{b_n}\le Cn^6e^{\Delta_n^-}=n^{306+o(1)},
\]
and
\[
 \deg T_n\le3m_n=O(n^{1/4}\log n).
\]
Combining these facts with \eqref{eq:Tcoef}, and taking logarithms of
the right-hand side, gives only
$O_m(n^{1/4}(\log n)^2)$.  Hence
\begin{equation}\label{eq:Tderivs}
 \|T_n(b_n\cdot)\|_{C^m(b_n^{-1}U_n)}
 \le \exp\bigl(C_mn^{1/4}(\log n)^2\bigr)
\end{equation}
for every fixed $m$.

We also need derivative bounds for $v_n$.  On the support of its cutoff
inside $U_n$, write $\xi=s/t$.  Since $|\pi_n|\le2$ there,
\[
 |f(\xi)|=\frac{|t\pi_n|}{b_n}
 \le\frac{8S_n}{b_n}\le n^{307}
\]
for large $n$.  From $|\xi(\xi-1)|\le n^{307}$ we obtain
$|\xi|\le Cn^{154}$.  Lemma~\ref{lem:scalar} therefore gives
\begin{equation}\label{eq:vsmall}
 \|v_n\|_{C^0(U_n)}\le n^C.
\end{equation}

For positive derivatives it is important to work in the rescaled
variables.  There $v_n(b_nx)$ vanishes whenever
\[
 |x_2|\le\frac{\varepsilon_n}{b_n}
 =\frac12(2a_n)^{-38}.
\]
Thus, on the region where derivatives can be nonzero,
\[
 |x_2|^{-1}\le2(2a_n)^{38}=2n^{76}.
\]
Moreover $\|x\|\le4S_n/b_n\le n^C$.  The two rational functions which
occur in \eqref{eq:vk} become
\[
 \frac{x_1}{2a_nx_2},\qquad
 \frac{x_1(x_1-x_2)}{x_2^3}
 =\frac{x_1^2}{x_2^3}-\frac{x_1}{x_2^2}.
\]
After differentiating a fixed number $m$ of times, every term is a
constant multiple of
\[
 (2a_n)^{-1}\frac{x_1^j}{x_2^p}
 \quad\text{or}\quad
 \frac{x_1^j}{x_2^p},
\]
with $0\le j\le2$ and $p\le m+3$.  The preceding polynomial bounds on
$\|x\|$ and $|x_2|^{-1}$ show that all these derivatives are bounded by
$n^{C_m}$.  The derivatives of $\rho_n=h_{a_n}$ are polynomially bounded
by \eqref{eq:scalarbounds}.  The product and chain rules therefore give
\begin{equation}\label{eq:vderivs}
 \|v_n(b_n\cdot)\|_{C^m(b_n^{-1}U_n)}
 \le n^{C_m}
\end{equation}
for every fixed $m\ge0$.

We can now correct the Levi form.  Define, away from the zero set of
$T_n$,
\[
 \psi_n(z)=\frac{\|z\|^2}{S_n^2}+\log|T_n(z)|^2,
\]
and set
\[
 q_n(z)=
 \begin{cases}
    \chi(\psi_n(z)),&T_n(z)\ne0,\\
    0,&T_n(z)=0.
 \end{cases}
\]
The function $\psi_n$ is plurisubharmonic off $\{T_n=0\}$: the first
term is strictly plurisubharmonic and $\log|T_n|^2$ is
plurisubharmonic.  Since $\chi$ is convex and nondecreasing,
$\chi\circ\psi_n$ is plurisubharmonic there.

Near every zero of $T_n$, the logarithm tends to $-\infty$, so
$\psi_n<-1$ on a punctured neighborhood.  Because $\chi=0$ on
$(-\infty,-1]$, $q_n$ is identically zero on a full neighborhood of
each zero.  This proves smoothness across the zero set and also shows
that $q_n$ vanishes near every $K_\ell$ with $\ell\in I_n$.
Furthermore, on $B(0,r_n)$, estimate \eqref{eq:Tsmall} gives
$\log|T_n|^2\le-4$, while $r_n<S_n$ gives
$\|z\|^2/S_n^2<1$.  Hence $\psi_n<-3<-1$ and
\begin{equation}\label{eq:qsmall}
 q_n=0\quad\text{on }B(0,r_n).
\end{equation}

At a point where $v_n$ has a negative Levi direction,
Lemma~\ref{lem:Tlower} gives $|T_n|\ge e^4$.  Thus
\[
 \psi_n\ge\log|T_n|^2\ge8>1,
\]
so $\chi(\psi_n)=\psi_n$ in a neighborhood of that point.  Since
$\log|T_n|^2$ is pluriharmonic where $T_n\ne0$, the Levi form there is
exactly
\begin{equation}\label{eq:qlevi}
 \LL_{q_n}(z;X)=S_n^{-2}\|X\|^2.
\end{equation}
On the other hand, \eqref{eq:vderivs} with $m=2$, after returning from
$x$ to $z=b_nx$, shows that the negative part of the Levi form of $v_n$
on $U_n$ is bounded by
\[
 Cb_n^{-2}n^C\|X\|^2.
\]
To dominate this by $\lambda_n\LL_{q_n}$ it is enough to require
\[
 \lambda_nS_n^{-2}\ge Cb_n^{-2}n^C,
\]
or
\[
 \lambda_n\ge Cn^C\left(\frac{S_n}{b_n}\right)^2.
\]
Since $S_n/b_n\le n^{307}$, the right-hand side is still a fixed power
of $n$.  We may therefore choose
\begin{equation}\label{eq:lambdan}
 1\le\lambda_n\le n^C
\end{equation}
so that
\begin{equation}\label{eq:Hn}
 H_n=v_n+\lambda_nq_n
\end{equation}
is smooth and plurisubharmonic on $U_n$.  Because $q_n$ vanishes near
all $K_\ell$, $\ell\in I_n$, the correction does not change the
prescribed restrictions there.  It also vanishes on $B(0,r_n)$.

We now give the value and derivative estimates needed later.  On
$U_n=B(0,4S_n)$ we have $\|z\|^2/S_n^2\le16$.  Moreover
\eqref{eq:Tderivs} with $m=0$ bounds $\log^+|T_n|$ by a fixed power of
$n$.  Since
\[
 0\le\chi(x)\le\max\{x,0\}+\frac12,
\]
it follows that
\begin{equation}\label{eq:q-bound}
 0\le q_n\le n^C,
 \qquad
 \|H_n\|_{C^0(U_n)}\le n^C.
\end{equation}

For higher derivatives, consider a point where $q_n$ or one of its
derivatives is nonzero.  Then $\psi_n>-1$; at $\psi_n=-1$ all
derivatives of $\chi(\psi_n)$ vanish because $\chi$ is identically zero
to the left of $-1$.  Since $\|z\|<4S_n$,
\[
 \log|T_n(z)|^2
 >-1-\frac{\|z\|^2}{S_n^2}>-17,
\]
and hence
\[
 |T_n(z)|>e^{-17/2}.
\]
On this set every derivative of $\log|T_n|^2$ of a fixed order is a
finite sum of products of derivatives of $T_n$ and $\overline{T_n}$,
divided by powers of $T_n$ and $\overline{T_n}$.  The denominators are
therefore uniformly bounded.  In rescaled coordinates,
\eqref{eq:Tderivs} and the chain rule give
\[
 \|q_n(b_n\cdot)\|_{C^m(b_n^{-1}U_n)}
 \le \exp\bigl(C_m n^{1/4}(\log n)^2\bigr).
\]
Returning from $x$ to $z=b_nx$ multiplies an $m$th derivative by at
most
\[
 b_n^{-m}=e^{mn^{3/4}}.
\]
This term dominates the smaller
$\exp(C_mn^{1/4}(\log n)^2)$ contribution.  Together with
\eqref{eq:vderivs} and $\lambda_n\le n^C$, we conclude that, for every
fixed $m\ge0$,
\begin{equation}\label{eq:Hderivs}
 \|H_n\|_{C^m(U_n)}\le e^{C_m n^{3/4}}.
\end{equation}

\subsection{Extension to the whole space}
The function $H_n$ is defined only on $U_n=B(0,4S_n)$.  We now extend it
to $\C^2$ without changing it near the curves.  The extension uses a
smooth maximum with a logarithmic plurisubharmonic function.

Define
\[
 \mathcal M(x,y)=x+\chi(y-x).
\]
Because $\chi(u)=0$ for $u\le-1$ and $\chi(u)=u$ for $u\ge1$,
\[
 \mathcal M(x,y)=x\quad\text{if }y-x\le-1,
 \qquad
 \mathcal M(x,y)=y\quad\text{if }y-x\ge1.
\]
If $f$ and $g$ are smooth plurisubharmonic functions, then
\begin{align*}
 \LL_{\mathcal M(f,g)}(X)
 &=(1-\chi'(g-f))\LL_f(X)+\chi'(g-f)\LL_g(X)\\
 &\quad+\chi''(g-f)|\partial(g-f)(X)|^2\ge0.
\end{align*}
Thus $\mathcal M(f,g)$ is again plurisubharmonic.

Put
\[
 A_n=4\bigl(1+\|H_n\|_{C^0(U_n)}\bigr)\le n^C
\]
and, for $z\ne0$,
\[
 \ell_n(z)=A_n\log\frac{\|z\|}{2S_n}.
\]
Define
\begin{equation}\label{eq:uk}
 u_n(z)=
 \begin{cases}
  \mathcal M(H_n(z),\ell_n(z)),&0<\|z\|<4S_n,\\
  \ell_n(z),&\|z\|\ge4S_n,\\
  0,&z=0.
 \end{cases}
\end{equation}

\begin{proposition}\label{prop:uk}
For all sufficiently large $n$, $u_n\in C^\infty(\C^2)\cap
\psh(\C^2)$ and $u_n\ge -n^C$. Moreover,
\begin{equation}\label{eq:uk-restriction}
 u_n\circ\gamma_n=\rho_n,\qquad
 u_n|_{K_\ell}=0\quad(\ell>n),\qquad
 u_n=0\quad\text{on }B(0,r_n).
\end{equation}
For every fixed $R>0$, $m\ge0$,
\begin{equation}\label{eq:uk-Cm}
 \|u_n\|_{C^m(B(0,R))}\le e^{C_{m,R}n^{3/4}}.
\end{equation}
Furthermore, for $j>n$ we have
\begin{equation}\label{eq:later-on-earlier}
 0\le u_j\le j^C\qquad\text{on }K_n.
\end{equation}
\end{proposition}

\begin{proof}
We first check that the smooth maximum agrees with the correct branch
near both ends of the annulus where the two formulas in
\eqref{eq:uk} meet.

If $0<\|z\|\le S_n$, then
\[
 \ell_n(z)\le-A_n\log2.
\]
Therefore
\[
 \ell_n(z)-H_n(z)
 \le-A_n\log2+\|H_n\|_{C^0(U_n)}<-1
\]
by the definition $A_n=4(1+\|H_n\|_\infty)$.  Hence
\[
 u_n=H_n\qquad(0<\|z\|\le S_n).
\]
In particular $u_n$ vanishes on $B(0,r_n)$ because $H_n$ does, so the
definition extends smoothly across $z=0$.

If $3S_n\le\|z\|<4S_n$, then
\[
 \ell_n(z)\ge A_n\log\frac32.
\]
Thus
\[
 \ell_n(z)-H_n(z)
 \ge A_n\log\frac32-\|H_n\|_{C^0(U_n)}>1,
\]
and consequently $u_n=\ell_n$ on this annulus.  This proves that the
inner and outer formulas in \eqref{eq:uk} agree smoothly.  Since both
$H_n$ and $\ell_n$ are plurisubharmonic on their common domain,
$u_n$ is globally smooth and plurisubharmonic.  Also
$H_n=v_n+\lambda_nq_n\ge -\|v_n\|\ge-n^C$ on $U_n$, whereas
$\ell_n>0$ outside $B(0,3S_n)$.  Hence $u_n\ge-n^C$ everywhere.

We next prove the restriction identities.  If $\ell\in I_n$, then
$K_\ell\subset B(0,S_n)$, so
\[
 u_n|_{K_\ell}=H_n|_{K_\ell}=v_n|_{K_\ell}.
\]
Lemma~\ref{lem:local} therefore gives
$u_n\circ\gamma_n=\rho_n$ and $u_n|_{K_\ell}=0$ for
$\ell\in I_n$, $\ell\ne n$.  If $\ell>n+N_n$, then the scale comparison
$R_{n+N_n+1}<r_n$ and monotonicity of $R_j$ give
\[
 K_\ell\subset B(0,R_\ell)\subset B(0,r_n),
\]
where $u_n=0$.  Hence $u_n|_{K_\ell}=0$ for every $\ell>n$.
This proves \eqref{eq:uk-restriction}.

For the $C^m$ estimate, we already know from \eqref{eq:Hderivs} that
\[
 \|H_n\|_{C^m(U_n)}\le e^{C_mn^{3/4}}.
\]
On the annulus $S_n\le\|z\|\le4S_n$, differentiation of the logarithm
gives, for $m\ge1$,
\[
 |D^m\ell_n(z)|\le C_m A_nS_n^{-m}.
\]
Now
\[
 \log S_n^{-1}=O(n^{3/4}),
\]
so, for fixed $m$, $S_n^{-m}\le e^{C_mn^{3/4}}$.  The derivatives of
$\chi$ are bounded, hence the chain rule applied to the smooth maximum
produces the same exponential bound.  Outside $B(0,3S_n)$ we have
$u_n=\ell_n$.  On the part of a fixed ball $B(0,R)$ where this formula
is used, namely $3S_n\le\|z\|\le R$,
\[
 |\ell_n(z)|\le C_R A_n(1+\log S_n^{-1})\le n^C,
\]
while its positive-order derivatives satisfy the preceding bound.
This proves \eqref{eq:uk-Cm}.

Finally fix $j>n$ and $z\in K_n$.  Since $j\ne n$, Lemma~\ref{lem:local}
gives $v_j(z)=0$, and therefore
\[
 H_j(z)=\lambda_jq_j(z)\ge0.
\]
If $\|z\|<4S_j$, then the smooth maximum satisfies
\[
 \max\{H_j(z),\ell_j(z)\}
 \le u_j(z)
 \le\max\{H_j(z),\ell_j(z)\}+\frac12.
\]
The first term is nonnegative.  The value estimate for $H_j$ on $U_j$
and $\ell_j(z)\le A_j\log2$ give $u_j(z)\le j^C$.  If instead
$\|z\|\ge4S_j$, then $u_j(z)=\ell_j(z)\ge0$.  Since
$K_n\subset B(0,1/2)$ for all sufficiently large $n$,
\[
 \ell_j(z)
 \le A_j\log\frac{1}{4S_j}
 \le j^C.
\]
This proves \eqref{eq:later-on-earlier}.
\end{proof}

\subsection{The infinite series}
The functions $u_n$ now have the two properties needed for summation.
For every fixed derivative order, their $C^m$ norms grow like
$e^{C_mn^{3/4}}$, which is subexponential in $n$ and hence dominated by
$2^n$.  At the same time, on each curve $K_n$, a long consecutive block
of the other functions vanishes exactly.  The latter property is what
makes the restriction error smaller than the principal term.

Fix $K$ sufficiently large for all preceding constructions and scale
comparisons, and also so that
\[
 K_n\subset B(0,1/2)\qquad(n\ge K).
\]
The index $K$ is chosen once and for all; it does not depend on the
derivative order.  For fixed $m$ and $R$, the threshold after which
\eqref{eq:uk-Cm} holds may depend on $m,R$, but discarding finitely many
terms never affects convergence.  Define
\begin{equation}\label{eq:u0}
 W_n=2^{-n}u_n,\qquad
 u_0=\sum_{n\ge K}W_n.
\end{equation}

\begin{proposition}\label{prop:interpolation}
The series \eqref{eq:u0} converges in $C^\infty$ on compact
subsets of $\C^2$. Its sum is smooth, plurisubharmonic and
flat at $0$. Furthermore,
\begin{equation}\label{eq:interpolation}
 \sup_{\zeta\in\overline\D}
 |u_0(\gamma_n(\zeta))-2^{-n}\rho_n(\zeta)|=o(2^{-n}).
\end{equation}
\end{proposition}

\begin{proof}
Fix a ball $B(0,R)$ and an integer $m\ge0$.  By
\eqref{eq:uk-Cm},
\[
 \|W_n\|_{C^m(B(0,R))}
 \le\exp\bigl(-n\log2+C_{m,R}n^{3/4}\bigr).
\]
Since $n^{3/4}=o(n)$, the exponent is at most
$-\frac12n\log2$ for all sufficiently large $n$.  Thus the series of
$C^m$ norms converges.  As this is true for every fixed $m,R$, the sum
converges in $C^\infty$ on compact subsets.  Every $W_n$ is
plurisubharmonic, so the locally uniform limit is plurisubharmonic.
Moreover each $u_n$ vanishes on a neighborhood of the origin.  Termwise
differentiation of the convergent series therefore shows that every
derivative of $u_0$ at the origin is zero: $u_0$ is flat there.

It remains to estimate the restriction to $K_n$.  There are four ranges
of indices.  If $j<n$, then $u_j|_{K_n}=0$ by
\eqref{eq:uk-restriction}.  For $j=n$,
\[
 W_n\circ\gamma_n=2^{-n}\rho_n.
\]
Now suppose
\[
 n<j\le n+M_n.
\]
The function $M_j=\lceil400j^{1/4}\log j\rceil$ is eventually
nondecreasing, so $M_j\ge M_n$.  Therefore
\[
 j-M_j\le j-M_n\le n,
\]
and also $n<j<j+N_j$.  Hence $n\in I_j$.  Proposition~\ref{prop:uk}
then gives $u_j|_{K_n}=0$.  Thus only the tail $j>n+M_n$ contributes to
the error.

Using \eqref{eq:later-on-earlier},
\begin{align*}
 \frac{\sup_{K_n}|u_0-W_n|}{2^{-n}}
 &\le\sum_{j>n+M_n}2^{-(j-n)}j^C.
\end{align*}
For a polynomially growing sequence, a geometric tail is controlled by
its first term up to another polynomial factor, so
\[
 \sum_{j>n+M_n}2^{-(j-n)}j^C
 \le C(n+M_n)^C2^{-M_n}\longrightarrow0.
\]
This is exactly \eqref{eq:interpolation}.
\end{proof}

\section{The domain and its Kobayashi incompleteness}
We now assemble the preceding functions into a single smooth
plurisubharmonic potential and use the prescribed discs to produce a
finite-length chain approaching the boundary.  There are two points to
keep separate.  First, the global geometry of the domain must be
controlled: boundedness, smooth boundary, and pseudoconvexity.  Second,
the local interpolation error on the $n$th disc must be small compared
with the depth parameter $2^{-n}$ so that the corresponding analytic
disc really lies inside the domain.

Let $u_0$ be the function obtained in Proposition~\ref{prop:interpolation}.
By Proposition~\ref{prop:uk},
\[
 u_n\ge-n^C\quad\text{on }\C^2,
\]
so
\[
 u_0=\sum_{n\ge K}2^{-n}u_n
\]
is bounded below.  Increase the initial index $K$, if necessary, so
that $4S_n<1$ for every $n\ge K$.  Then for every $\|z\|\ge1$ we are in
the outer region of \eqref{eq:uk} for every $n$, and therefore
\[
 u_n(z)=A_n\log\frac{\|z\|}{2S_n}.
\]
Summing the series gives
\[
 u_0(z)=A\log\|z\|+B\qquad(\|z\|\ge1),
\]
where
\begin{equation}\label{eq:outside}
 \begin{aligned}
 A&=\sum_{n\ge K}2^{-n}A_n>0,\\
 B&=-\sum_{n\ge K}2^{-n}A_n\log(2S_n).
 \end{aligned}
\end{equation}
Both series converge.  Indeed $A_n\le n^C$, while
$\log(1/S_n)=O(n^{3/4})$, so the summands are a polynomial factor times
$2^{-n}$.

Choose $\mu>0$ so small that
\[
 \mu u_0(z)<1\qquad(\|z\|\le1),
\]
and set
\begin{equation}\label{eq:hartogs}
 \rho(s,t)=\mu u_0(s,t),\qquad
 \Omega=\{(s,t,w)\in\C^3:|w|^2+\rho(s,t)<1\}.
\end{equation}
We first verify the global properties of this domain.

Put
\[
 R=\exp\frac{1/\mu-B}{A}.
\]
Because $\mu u_0<1$ on $\|z\|\le1$ and the outer formula gives
$u_0(z)=A\log\|z\|+B$ at $\|z\|=1$, we have $\mu B<1$ and hence
$R>1$.  For $\|z\|\ge1$,
\[
 \rho(z)<1
 \quad\Longleftrightarrow\quad
 \mu(A\log\|z\|+B)<1
 \quad\Longleftrightarrow\quad
 \|z\|<R.
\]
Inside the unit ball the inequality holds by our choice of $\mu$.
Consequently
\begin{equation}\label{eq:base}
 \{(s,t)\in\C^2:\rho(s,t)<1\}=B(0,R).
\end{equation}
Thus the base of the Hartogs domain is bounded.  Since each fiber in the
$w$-variable is a disc centered at the origin, $\Omega$ is connected.
It is bounded as well: $\|(s,t)\|<R$, and since $u_0$ is bounded below,
\[
 |w|^2<1-\mu\inf_{\C^2}u_0.
\]

The function
\[
 r(s,t,w)=|w|^2+\rho(s,t)-1
\]
is a global smooth plurisubharmonic defining function.  We only need to
check that its differential does not vanish on the zero set.  If
$w\ne0$, then
\[
 \frac{\partial r}{\partial w}=\overline w\ne0.
\]
If $w=0$ and $r=0$, then \eqref{eq:base} gives $\|(s,t)\|=R>1$.
In that region
\[
 \rho(z)=\mu(A\log\|z\|+B),
\]
so, writing $z=(z_1,z_2)=(s,t)$,
\[
 \frac{\partial r}{\partial z_j}(z,0)
 =\frac{\mu A\overline z_j}{2\|z\|^2},
 \qquad j=1,2.
\]
At least one of these derivatives is nonzero.  Hence $dr\ne0$ on
$\partial\Omega$.  We have proved that $\Omega$ is bounded,
$C^\infty$-smoothly bounded, and pseudoconvex.

We now insert the analytic discs.  Set
\[
 \delta_n=\mu2^{-n}.
\]
Proposition~\ref{prop:interpolation} says that, uniformly for
$\zeta\in\overline\D$,
\[
 u_0(\gamma_n(\zeta))
 =2^{-n}\rho_n(\zeta)+o(2^{-n}).
\]
Multiplying by $\mu$ gives
\begin{equation}\label{eq:u-hartogs}
 \rho(\gamma_n(\zeta))
 =\delta_n\bigl(\rho_n(\zeta)+o(1)\bigr).
\end{equation}
The normal component of the disc is chosen to have precisely the same
scale $\delta_n$:
\begin{equation}\label{eq:normal-discs}
 \varphi_n(\zeta)=
 \left(\gamma_n(\zeta),
 1+\frac{\delta_n}{2}(-1+a_n\zeta)\right),
 \qquad\zeta\in\D.
\end{equation}
We check directly that $\varphi_n(\D)\subset\Omega$ for all large $n$.
Write
\[
 L_n(\zeta)=-1+a_n\zeta.
\]
Then
\begin{align*}
 \left|1+\frac{\delta_n}{2}L_n(\zeta)\right|^2-1
 &=\delta_n\Ree L_n(\zeta)
   +\frac{\delta_n^2}{4}|L_n(\zeta)|^2\\
 &=\delta_n(-1+a_n\Ree\zeta)
   +\frac{\delta_n^2}{4}|-1+a_n\zeta|^2.
\end{align*}
Combining this identity with \eqref{eq:u-hartogs},
\begin{align*}
 r(\varphi_n(\zeta))
 &=\delta_n\bigl(-1+a_n\Ree\zeta
   +\rho_n(\zeta)+o(1)\bigr)
   +\frac{\delta_n^2}{4}|-1+a_n\zeta|^2.
\end{align*}
Lemma~\ref{lem:scalar}, applied with $a=a_n$, gives
\[
 \rho_n(\zeta)\le1-a_n\Ree\zeta-\frac1{16}
 \qquad(|\zeta|\le1).
\]
The linear terms therefore cancel and leave a fixed negative margin:
\[
 -1+a_n\Ree\zeta+\rho_n(\zeta)\le-\frac1{16}.
\]
Moreover $|-1+a_n\zeta|=O(a_n)$ uniformly on $\overline\D$.  Hence
\[
 r(\varphi_n(\zeta))
 \le-\frac{\delta_n}{16}+o(\delta_n)
       +O(\delta_n^2a_n^2).
\]
Now
\[
 \delta_na_n^2
 =\mu2^{-n}\frac{n^4}{4}\longrightarrow0,
\]
so the last term is $o(\delta_n)$.  Therefore
\[
 r(\varphi_n(\zeta))<0
\]
uniformly on $\overline\D$ for all sufficiently large $n$.  This proves
$\varphi_n(\D)\subset\Omega$.

The reason for using the nodal curve is now visible.  Since
$\gamma_n(0)=\gamma_n(\tau_n)=0$, both endpoints of the disc lie on the
same complex normal line.  At $\zeta=0$,
\[
 \varphi_n(0)=\left(0,0,1-\frac{\delta_n}{2}\right).
\]
At $\zeta=\tau_n=(2a_n)^{-1}$ we have $a_n\tau_n=1/2$, so
\[
 \varphi_n(\tau_n)
 =\left(0,0,1+\frac{\delta_n}{2}\left(-1+\frac12\right)\right)
 =\left(0,0,1-\frac{\delta_n}{4}\right).
\]
Since $\delta_{n+1}=\delta_n/2$,
\[
 \varphi_n(\tau_n)=\varphi_{n+1}(0).
\]
Thus the discs concatenate exactly; no interpolation error enters the
endpoint matching.

Finally,
\[
 \tau_n=(2a_n)^{-1}=n^{-2}.
\]
Holomorphic maps decrease the Kobayashi distance, so
\[
 k_\Omega\bigl(\varphi_n(0),\varphi_{n+1}(0)\bigr)
 \le\rho_\D(0,\tau_n)
 =\operatorname{arctanh}(n^{-2}).
\]
Therefore
\[
 \sum_{n\ge n_0}
 k_\Omega\bigl(\varphi_n(0),\varphi_{n+1}(0)\bigr)
 \le\sum_{n\ge n_0}\operatorname{arctanh}(n^{-2})<\infty.
\]
The endpoint sequence is consequently Kobayashi-Cauchy.  In Euclidean
coordinates it is simply
\[
 \left(0,0,1-\frac{\mu}{2^{n+1}}\right),
\]
so it converges to $(0,0,1)\in\partial\Omega$.  Since $u_0$ is flat at
the origin, $d\rho(0)=0$, and therefore
\[
 dr_{(0,0,1)}=2\,d(\Ree w).
\]
Thus the endpoint sequence approaches along the inward real normal.

It remains only to rule out the possibility that this Kobayashi-Cauchy
sequence converges, in the Kobayashi metric, to an interior point.  Let
$\Omega\subset B(0,M)$.  Each coordinate function divided by $2M$ maps
$\Omega$ holomorphically into $\D$.  By Schwarz--Pick,
Kobayashi convergence in $\Omega$ implies convergence of every
coordinate, hence Euclidean convergence.  Our sequence has Euclidean
limit $(0,0,1)\notin\Omega$, so it has no Kobayashi limit in $\Omega$.
The Kobayashi distance is therefore incomplete.

The same sequence is Carath\'eodory-Cauchy because
$c_\Omega\le k_\Omega$.  Conversely the coordinate functions divided
by $2M$ belong to the defining family for the Carath\'eodory distance,
so Carath\'eodory convergence also implies Euclidean convergence.
Therefore the sequence has no Carath\'eodory limit in $\Omega$, and the
Carath\'eodory distance is incomplete as well.

Finally, because $\rho$ is flat at $(0,0)$, the holomorphic curve
\[
 \zeta\longmapsto(\zeta,0,1)
\]
has infinite order of contact with $\partial\Omega$ at $(0,0,1)$.
Indeed
\[
 r(\zeta,0,1)=\rho(\zeta,0)
\]
is flat at $\zeta=0$.  Thus the boundary point responsible for
incompleteness has infinite D'Angelo type.  This completes the proof of
Theorem~\ref{thm:main}.

\begin{remark}\label{rem:dichotomy}
The construction also shows concretely where the finite-type
self-improvement argument breaks down.  For the centers
\[
 z_n=\varphi_n(0)=\left(0,0,1-\frac{\delta_n}{2}\right),
 \qquad \delta_n=\mu2^{-n},
\]
we have $\rho(0)=0$, and hence the actual defining-function depth is
\[
 d_n:=-r(z_n)
 =1-\left(1-\frac{\delta_n}{2}\right)^2
 =\delta_n-\frac{\delta_n^2}{4}
 \asymp\delta_n.
\]
Since $2a_n=n^2$ and $\Gamma'(0)=(0,-1)$,
\[
 \gamma_n'(0)=(0,-b_nn^2),
 \qquad b_n=e^{-n^{3/4}}.
\]
Therefore
\[
 \|\varphi_n'(0)\|\ge b_nn^2=n^2e^{-n^{3/4}},
\]
and for every fixed $\varepsilon>0$,
\[
 \frac{\|\varphi_n'(0)\|}{d_n^\varepsilon}
 \ge c_\varepsilon n^2
 \exp\!\left(\varepsilon n\log2-n^{3/4}\right)
 \longrightarrow\infty.
\]
Thus, for every fixed $\varepsilon>0$, the centers of these discs violate every estimate of the form
\[
 \|f'(0)\|\le C(-r(f(0)))^\varepsilon.
\]
This is precisely
the first step of the finite-type self-improvement that fails in the
present infinite-type example.  At the same time,
\[
 \left|\sum_{j=1}^N r_{z_j}(z_n)(\varphi_n)'_j(0)\right|
 =\left(1-\frac{\delta_n}{2}\right)\frac{\delta_na_n}{2}
 \asymp \delta_n n^2
 \asymp \delta_n\log^2\frac1{\delta_n},
\]
whereas Proposition~\ref{prop:self-general} gives an upper bound of order
$\delta\log(1/\delta)$ for the normal derivative of an analytic disc,
and the sharper estimate in Remark~\ref{rem:sharp-normal} improves this to
$\delta\log(1/\delta)/\log\log(1/\delta)$.
The same contrast is visible at the distance level: here
\[
 k_\Omega(z_n,z_{n+1})\le\operatorname{arctanh}(n^{-2})=O(n^{-2}),
\]
while the first completeness estimate gives, for dyadically separated normal
depths in a finite-type domain, a lower bound of order $1/n$, and the sharper
estimate gives a lower bound of order $(\log n)/n$.
\end{remark}

\medskip
\noindent\textbf{Acknowledgements.}
The author thanks Ahmed Yekta \"{O}kten, Gautam Bharali, and
Nikolai Nikolov for discussions concerning finite-type estimates and
questions arising in the proof.

OpenAI ChatGPT was used to improve the exposition of the manuscript and
to assist in developing some of the mathematical arguments. The author
takes full responsibility for the mathematical content of the paper.

\end{document}